\documentclass[a4paper]{amsart}
\usepackage{amsmath,amsthm,amssymb,latexsym,epic,bbm}
\usepackage{graphicx,enumerate,stmaryrd}
\usepackage[all]{xy}

\newtheorem{theorem}{Theorem}
\newtheorem{lemma}[theorem]{Lemma}

\newtheorem{corollary}[theorem]{Corollary}
\newtheorem{proposition}[theorem]{Proposition}
\newtheorem{example}[theorem]{Example}

\newcommand{\tto}{\twoheadrightarrow}

\font\sc=rsfs10
\newcommand{\cC}{\sc\mbox{C}\hspace{1.0pt}}
\newcommand{\cR}{\sc\mbox{R}\hspace{1.0pt}}
\newcommand{\cI}{\sc\mbox{I}\hspace{1.0pt}}
\newcommand{\cS}{\sc\mbox{S}\hspace{1.0pt}}

\newcommand{\cA}{\sc\mbox{A}\hspace{1.0pt}}

\font\scc=rsfs7
\newcommand{\ccC}{\scc\mbox{C}\hspace{1.0pt}}

\begin{document}

\title[Cell  $2$-representations]{Cell
$2$-representations  of finitary $2$-categories}
\author{Volodymyr Mazorchuk and Vanessa Miemietz}
\date{\today}

\begin{abstract}
We study $2$-representations of finitary $2$-categories with involution 
and adjunctions by functors on module categories over finite dimensional
algebras. In particular, we define, construct and describe in detail 
(right) cell $2$-representations inspired by Kazhdan-Lusztig cell modules 
for Hecke algebras. Under some natural assumptions we show that cell
$2$-representations are strongly simple and do not depend on the choice 
of a right cell inside a two-sided cell. This reproves and extends 
the uniqueness result on categorification of Kazhdan-Lusztig cell modules 
for Hecke algebras of type $A$ from \cite{MS2}.
\end{abstract}

\maketitle

\section{Introduction and description of the results}\label{s0}

The philosophy of categorification, which originated in work
of Crane and Frenkel (see \cite{Cr,CF}) some fifteen years ago,
is nowadays usually formulated in terms of $2$-categories.
A {\em categorification} of an algebra (or category) $A$ is now usually
understood as a $2$-category $\cA$, whose decategorification 
is $A$. Therefore a natural problem is to ``upgrade'' the
representation theory of $A$ to a $2$-representation theory of $\cA$. 
The latter philosophy has been propagated by Rouquier in \cite{Ro0,Ro} 
based on the earlier development in \cite{CR}.

Not much is known about the $2$-category of $2$-representations
of an abstract $2$-category. Some $2$-representations of $2$-categories
categorifying Kac-Moody algebras were constructed and studied in 
\cite{Ro}. On the other hand, there are many examples of 
$2$-representations of various $2$-categories in the literature, 
sometimes without an explicit emphasis on their $2$-categorical 
nature, see for example \cite{Kv,St,KMS,MS2,KL} and references therein.
A different direction of the representation theory of certain
classes of $2$-categories was investigated in \cite{EO,EGNO}.

$2$-categorical philosophy also appears, in a disguised form, in \cite{Kh}. 
In this article the author defines so-called ``categories with
full projective functors'' and considers ``functors naturally
commuting with projective functors''. The former can be understood as
certain ``full'' $2$-representations of a $2$-category and the latter as 
morphisms between these $2$-representations.

The aim of the present article is to look at the study of 
$2$-representations of abstract $2$-categories from a  more systematic 
and more abstract prospective. Given an algebra $A$ there are two natural 
ways to construct $A$-modules. The first way is to fix a presentation for 
$A$ and construct $A$-modules using generators and checking relations.
The second way is to look at homomorphisms between free $A$-modules 
and construct their cokernels. Rouquier's approach to $2$-representation 
theory from \cite{Ro0,Ro} goes along the first way. In the
present article, we try the second one.

Our main object of study is what we call a {\em fiat category} $\cC$, that 
is a (strict) $2$-category with involution which has finitely many 
objects, finitely many isomorphism classes of indecomposable $1$-morphisms, and 
finite dimensional spaces of $2$-morphisms that are also supposed to
contain adjunction morphisms. Our $2$-setup is described in detail 
in Section~\ref{s1}. In Section~\ref{s2} we study {\em principal
$2$-representations} of fiat categories, which are analogues of 
indecomposable projective modules over an algebra. We give an 
explicit construction of principal $2$-representations and prove a 
natural analogue of the universal property for them. Adding up all 
principal representations we obtain the regular $\cC$-bimodule, which 
gives rise to an abelian $2$-category $\hat{\cC}$ enveloping 
the original category $\cC$. The category $\hat{\cC}$ is no longer 
fiat, but has the advantage of being abelian. We show that every 
$2$-representation of $\cC$ extends to a $2$-representation of
$\hat{\cC}$ in a natural way.

Inspired by Kazhdan-Lusztig combinatorics (see \cite{KaLu}), in Section~\ref{s3} we define, for every 
fiat category $\cC$, the concepts of left, right and 
two-sided cells and cell $2$-representations associated with right 
cells. We expect cell representations to be the most natural candidates 
for ``simple'' $2$-representations, whatever the latter could possibly 
mean (which is still unclear). We describe the algebraic structure of 
module categories on which a cell $2$-representation operates and 
determine homomorphisms from a cell $2$-representation. 
We also study in detail the combinatorial 
structure of two natural classes of cells, which we call regular
and strongly regular. These turn out to have particularly nice properties 
and appear in many natural examples.

Section~\ref{s5} is devoted to the study of the local structure of cell 
$2$-representations. We show that the essential part of cell 
$2$-representations is governed by the action of $1$-morphisms from the
associated two-sided cell and describe algebraic properties of cell 
$2$-representations in terms of the cell combinatorics of this two-sided 
cell.

In Section~\ref{s4} we define and study the notions of cyclicity and 
strong simplicity for $2$-representations. A $2$-representation is 
called cyclic if it is generated, in the $2$-categorical sense of 
categories with full projective functors in \cite{Kh}, by some object 
$M$. This means that the natural map from $\hat{\cC}$ to our
$2$-representation, sending $\mathrm{F}$ to $\mathrm{F}\,M$ is 
essentially surjective on objects and surjective on morphisms. 
A $2$-representation is called
strongly simple  if it is generated, in the $2$-categorical
sense, by any simple object. We show that all cell
$2$-representations are cyclic and prove the following main result:

\begin{theorem}\label{thmmain}
Let $\cC$ be a fiat category. Then, under some natural technical assumptions, 
every cell $2$-representation of $\cC$
associated with a strongly regular right cell is strongly simple. 
Moreover, under the same assumptions,  every two cell 
$2$-representations of $\cC$ associated with strongly regular right cells 
inside the same two-sided cell are equivalent.
\end{theorem}

Finally, in Section~\ref{s6} we give several examples. The prime
example is the fiat category of projective functors acting on the 
principal block (or a direct sum of some, possibly singular, blocks)
of the BGG category $\mathcal{O}$ for a semi-simple complex finite 
dimensional Lie algebra. This example is given by Kazhadan-Lusztig 
combinatorics and our cells coincide with the classical 
Kazhdan-Lusztig cells. As an application of Theorem~\ref{thmmain}
we reprove, extend and strengthen the uniqueness result on categorification 
of Kazhdan-Lusztig cell modules for Hecke algebras of type $A$ from 
\cite{MS2}. We also present another example of a fiat category 
$\cC_A$ given by projective endofunctors of the module category of a 
weakly symmetric self-injective finite dimensional associative algebra $A$.
We show that the latter example is ``universal'' in the sense that, 
under the same assumptions as mentioned in 
Theorem~\ref{thmmain}, every cell $2$-representation of a fiat
category gives rise to a $2$-functor to some $\cC_A$. 
\vspace{0.2cm}

\noindent
{\bf Acknowledgments.} The first author was partially supported
by the Swedish Research Council. A part of this work was done
during a visit of the second author to Uppsala University, which
was supported by the Faculty of Natural Science of Uppsala University.
The financial support and hospitality of Uppsala University
are gratefully acknowledged. We thank Catharina Stroppel 
and Joseph Chuang for stimulating discussions.

\section{$2$-setup}\label{s1}

\subsection{Notation}\label{s1.1}

For a $2$-category $\cC$, objects of $\cC$ will be denoted by 
$\mathtt{i},\mathtt{j}$ and so on. For $\mathtt{i},\mathtt{j}\in \cC$, 
objects of $\cC(\mathtt{i},\mathtt{j})$ ($1$-morphisms of $\cC$) will 
be called $\mathrm{F},\mathrm{G}$ and so on. For 
$\mathrm{F},\mathrm{G}\in \cC(\mathtt{i},\mathtt{j})$, morphisms 
from $\mathrm{F}$ to $\mathrm{G}$ ($2$-morphisms of $\cC$) will 
be written $\alpha,\beta$ and so on. 
The identity $1$-morphism in $\cC(\mathtt{i},\mathtt{i})$
will be denoted $\mathbbm{1}_{\mathtt{i}}$ and the identity 
$2$-morphism from $\mathrm{F}$ to $\mathrm{F}$ will be denoted
$\mathrm{id}_{\mathrm{F}}$. Composition of $1$-morphisms will be
denoted by $\circ$, horizontal composition of $2$-morphisms will be
denoted by $\circ_0$ and vertical composition of $2$-morphisms will be
denoted by $\circ_1$. We often abbreviate 
$\mathrm{id}_{\mathrm{F}}\circ_0\alpha$ and 
$\alpha\circ_0\mathrm{id}_{\mathrm{F}}$
by $\mathrm{F}(\alpha)$ and $\alpha_{\mathrm{F}}$, respectively.

For the rest of the paper we fix an algebraically closed field $\Bbbk$. 
As we will often consider categories $\cC(\mathtt{i},\mathtt{j})$, we will denote the morphism space between $X$ and $Y$ in such a category by 
$\mathrm{Hom}_{\cC(\mathtt{i},\mathtt{j})}(X,Y)$ to avoid the awkward looking $\cC(\mathtt{i},\mathtt{j})(X,Y)$.

\subsection{Finitary $2$-categories and $2$-representations}\label{s1.2}

In what follows, by a {\em $2$-category} we always mean a 
{\em strict} $2$-category and use the name {\em bicategory} for the 
corresponding non-strict structure. Note that any bicategory is 
biequivalent to a $2$-category (see, for example, \cite[2.3]{Le}).

We define a $2$-category $\cC$ to be {\em $\Bbbk$-finitary} provided that
\begin{enumerate}[(I)]
\item\label{fin.1} $\cC$ has finitely many objects;
\item\label{fin.2} for every $\mathtt{i},\mathtt{j}\in \cC$
the category $\cC(\mathtt{i},\mathtt{j})$ is a fully additive
(i.e. karoubian) $\Bbbk$-linear category with finitely many isomorphism 
classes of indecomposable objects, moreover, horizontal composition of
$1$-morphisms is biadditive;
\item\label{fin.4} for every $\mathtt{i}\in \cC$ the object
$\mathbbm{1}_{\mathtt{i}}\in\cC(\mathtt{i},\mathtt{i})$ is indecomposable.
\end{enumerate}
From now on $\cC$ will always be a $\Bbbk$-finitary $2$-category. 

Denote by $\mathfrak{R}_{\Bbbk}$ the $2$-category whose objects
are categories equivalent to module categories of finite-dimensional 
$\Bbbk$-algebras, $1$-morphisms are functors between objects, and 
$2$-morphisms are natural transformations of functors. We will understand 
a {\em $2$-representation} of $\cC$ to be a strict $2$-functor from 
$\cC$ to $\mathfrak{R}_{\Bbbk}$. By \cite[2.0]{Le}, 
$2$-representations of $\cC$, together
with strict $2$-natural transformations 
(i.e. morphisms between $2$-representation, given by
a collection of functors) and modifications
(i.e. morphisms between strict $2$-natural transformations, given by 
natural transformations between the defining functors), form
a strict $2$-category, which we denote by $\cC\text{-}\mathfrak{mod}$.
For simplicity we will identify 
objects in $\cC(\mathtt{i},\mathtt{j})$  with their images under 
a $2$-representation (i.e. we will use module notation).

\begin{example}\label{exm1}
{\rm 
Consider the algebra $D:=\mathbb{C}[x]/(x^2)$ of dual numbers. 
It is easy to check that the endofunctor 
$\mathrm{F}:=D\otimes_{\mathbb{C}}{}_-$
of $D\text{-}\mathrm{mod}$ satisfies
$\mathrm{F}\circ \mathrm{F}\cong \mathrm{F}\oplus \mathrm{F}$.
Therefore one can consider the $2$-category $\cS_2$ defined as
follows: $\cS_2$ has one object $\mathtt{i}:=D\text{-}\mathrm{mod}$;
$1$-morphisms of $\cS_2$ are all endofunctors of $\mathtt{i}$
which are isomorphic to a direct sum of copies of 
$\mathrm{F}$ and the identity functor; 
$2$-morphisms of $\cS_2$ are all natural transformations of functors.
The category $\cS_2$ is a $\mathbb{C}$-finitary $2$-category.
It comes together with the {\em natural}  representation 
(the embedding of $\cS_2$ into $\mathfrak{R}_{\mathbb{C}}$).
}
\end{example}

\subsection{Path categories associated to 
$\cC(\mathtt{i},\mathtt{j})$}\label{s1.3}

For $\mathtt{i},\mathtt{j}\in \cC$ let $\mathrm{F}_1,\mathrm{F}_2,
\dots,\mathrm{F}_r$ be a complete list of pairwise non-isomorphic 
indecomposable objects in $\cC(\mathtt{i},\mathtt{j})$. Denote by
$\mathcal{C}_{\mathtt{i},\mathtt{j}}$ the full subcategory of 
$\cC(\mathtt{i},\mathtt{j})$ with objects
$\mathrm{F}_1,\mathrm{F}_2,\dots,\mathrm{F}_r$.  As $\cC$ is 
$\Bbbk$-finitary, the path algebra of 
$\mathcal{C}_{\mathtt{i},\mathtt{j}}$ is a finite dimensional 
$\Bbbk$-algebra. There is a canonical equivalence between the
category  
$\mathcal{C}_{\mathtt{i},\mathtt{j}}^{\mathrm{op}}\text{-}\mathrm{mod}$
and the category of modules over the path algebra of 
$\mathcal{C}_{\mathtt{i},\mathtt{j}}$.

\begin{example}\label{exm2}
{\rm 
For the category $\cS_2$ from Example~\ref{exm1} the category
$\cS_2(\mathtt{i},\mathtt{i})$ has two indecomposable objects,
namely $\mathbbm{1}_{\mathtt{i}}$ and $\mathrm{F}$. Realizing exact functors
on $D\text{-}\mathrm{mod}$ as $D$-bimodules, the functor
$\mathbbm{1}_{\mathtt{i}}$ corresponds to the bimodule $D$ and the functor
$\mathrm{F}$ corresponds to the bimodule $D\otimes_{\mathbb{C}}D$.
Let $\alpha:D\otimes_{\mathbb{C}}D\to D$ be the unique morphism
such that $1\otimes 1\mapsto 1$; $\beta:D\to D\otimes_{\mathbb{C}}D$ 
be the unique morphism such that $1\mapsto 1\otimes x+x\otimes 1$; 
and $\gamma:D\otimes_{\mathbb{C}}D\to D\otimes_{\mathbb{C}}D$
be the unique morphism such that $1\otimes 1\mapsto 
1\otimes x-x\otimes 1$. Then it is easy to check that the category
$\mathcal{C}_{\mathtt{i},\mathtt{i}}$ is given by the
following quiver and relations:
\begin{displaymath}
\xymatrix{
\bullet\ar@/^/[rr]^{\alpha}\ar@(dl,ul)[]^{\gamma}
&& \bullet\ar@/^/[ll]^{\beta}
}, \quad
\begin{array}{c}
\gamma^2=-(\beta\alpha)^2, (\alpha\beta)^2=0,\\
\alpha\gamma=\gamma\beta=0.
\end{array}
\end{displaymath}
}
\end{example}

\subsection{$2$-categories with involution}\label{s1.4}

If $\cC$ is a $\Bbbk$-finitary $2$-category, then an {\em involution}
on $\cC$ is a lax involutive object-preserving anti-automorphism $*$
of $\cC$. A finitary $2$-category $\cC$ with involution $*$ is
said to have {\em adjunctions} provided that for any $\mathtt{i},
\mathtt{j}\in\cC$ and any $1$-morphism
$\mathrm{F}\in\cC(\mathtt{i},\mathtt{j})$ there exist 
$2$-morphisms $\alpha:\mathrm{F}\circ\mathrm{F}^*\to
\mathbbm{1}_{\mathtt{j}}$ and $\beta:\mathbbm{1}_{\mathtt{i}}\to
\mathrm{F}^*\circ\mathrm{F}$ such that 
$\alpha_{\mathrm{F}}\circ_1\mathrm{F}(\beta)=\mathrm{id}_{\mathrm{F}}$ and
$\mathrm{F}^*(\alpha)\circ_1\beta_{\mathrm{F}^*}=\mathrm{id}_{\mathrm{F}^*}$.
A $\Bbbk$-finitary $2$-category with an involution and adjunctions will
be called a {\em fiat category}.

\begin{example}\label{exm2.1}
{\rm 
The category $\cS_2$ from Example~\ref{exm1} is easily seen to be 
a fiat category.
}
\end{example}

\section{Principal $2$-representations}\label{s2}

\subsection{$2$-representations $\mathbf{P}_{\mathtt{i}}$}\label{s2.1}

Let $\cC$ be a finitary $2$-category. For $\mathtt{i},\mathtt{j}\in\cC$
denote by $\overline{\cC}(\mathtt{i},\mathtt{j})$ the category 
defined as follows: Objects of $\overline{\cC}(\mathtt{i},\mathtt{j})$
are diagrams of the form $\xymatrix{
\mathrm{F}\ar[r]^{\alpha}&\mathrm{G}}$,
where $\mathrm{F},\mathrm{G}\in \cC(\mathtt{i},\mathtt{j})$
are $1$-morphisms and $\alpha$ is a $2$-morphism.
Morphisms of $\overline{\cC}(\mathtt{i},\mathtt{j})$ are equivalence 
classes of diagrams as given by the solid part of the following picture:
\begin{displaymath}
\xymatrix{
\mathrm{F}\ar[rr]^{\alpha}\ar[d]_{\beta}&&\mathrm{G}
\ar[d]^{\beta'}\ar@{.>}[dll]_{\xi}\\
\mathrm{F}'\ar[rr]^{\alpha'}&&\mathrm{G}'
},\qquad \mathrm{F},\mathrm{F}',\mathrm{G},\mathrm{G}'\in 
\cC(\mathtt{i},\mathtt{j}),
\end{displaymath}
modulo the ideal generated by all morphisms for which there exists
$\xi$ as shown by the dotted arrow above such that $\alpha'\xi=\beta'$.
As $\cC$ is finitary category, the category 
$\overline{\cC}(\mathtt{i},\mathtt{j})$ is abelian and equivalent to 
$\mathcal{C}_{\mathtt{i},\mathtt{j}}^{\mathrm{op}}\text{-}\mathrm{mod}$,
see \cite{Fr}.

For $\mathtt{i}\in\cC$ define the $2$-functor 
$\mathbf{P}_{\mathtt{i}}:\cC\to\mathfrak{R}_{\Bbbk}$ as follows: 
for $\mathtt{j}\in \cC$ set $\mathbf{P}_{\mathtt{i}}(\mathtt{j})=
\overline{\cC}(\mathtt{i},\mathtt{j})$. Further, for 
$\mathtt{k}\in\cC$ and $\mathrm{F}\in \cC(\mathtt{j},\mathtt{k})$
left horizontal composition with (the identity on) $\mathrm{F}$ defines 
a functor from $\overline{\cC}(\mathtt{i},\mathtt{j})$ to
$\overline{\cC}(\mathtt{i},\mathtt{k})$. We define this functor to be
$\mathbf{P}_{\mathtt{i}}(\mathrm{F})$. Given a $2$-morphism 
$\alpha:\mathrm{F}\to \mathrm{G}$, left horizontal composition 
with $\alpha$ gives a natural transformation from 
$\mathbf{P}_{\mathtt{i}}(\mathrm{F})$ to 
$\mathbf{P}_{\mathtt{i}}(\mathrm{G})$. We define this natural transformation 
to be $\mathbf{P}_{\mathtt{i}}(\alpha)$. From the
definition it follows that $\mathbf{P}_{\mathtt{i}}$ is a strict 
$2$-functor from $\cC$ to $\mathfrak{R}_{\Bbbk}$. The 
$2$-representation $\mathbf{P}_{\mathtt{i}}$ is called
the {\em $\mathtt{i}$-th principal} $2$-representation of $\cC$.

For $\mathtt{i},\mathtt{j}\in \cC$ and a $1$-morphism
$\mathrm{F}\in \cC(\mathtt{i},\mathtt{j})$ we denote by
$P_{\mathrm{F}}$ the projective object $0\to \mathrm{F}$ of 
$\overline{\cC}(\mathtt{i},\mathtt{j})$.

\subsection{The universal property of 
$\mathbf{P}_{\mathtt{i}}$}\label{s2.2}

\begin{proposition}\label{prop3}
Let $\mathbf{M}$ be a $2$-representation of $\cC$ and 
$M\in \mathbf{M}(\mathtt{i})$. 
\begin{enumerate}[$($a$)$]
\item \label{prop3.1}
For $\mathtt{j}\in \cC$ define the functor $\Phi^M_{\mathtt{j}}:
\overline{\cC}(\mathtt{i},\mathtt{j})\to \mathbf{M}(\mathtt{j})$ 
as follows: $\Phi^M_{\mathtt{j}}$ sends a diagram $\xymatrix{
\mathrm{F}\ar[r]^{\alpha}&\mathrm{G}}$ 
in $\overline{\cC}(\mathtt{i},\mathtt{j})$ to the cokernel
of $\mathbf{M}(\alpha)_M$.
Then $\Phi^M=(\Phi^M_{\mathtt{j}})_{\mathtt{j}\in\ccC}$ is the
unique morphism from $\mathbf{P}_{\mathtt{i}}$ to $\mathbf{M}$
sending $P_{\mathbbm{1}_{\mathtt{i}}}$ to $M$.
\item \label{prop3.2} The correspondence 
$M\mapsto \Phi^M$ is functorial.
\end{enumerate}
\end{proposition}

\begin{proof}
Claim \eqref{prop3.1} follows directly from $2$-functoriality
of $\mathbf{M}$. To prove claim \eqref{prop3.2} let $f:M\to M'$.  
Choose now any $\mathrm{F},\mathrm{G}\in\cC(\mathtt{i},\mathtt{j})$ 
and $\alpha:\mathrm{F} \to \mathrm{G}$. Applying $\mathbf{M}$ to
$\mathrm{F}\overset{\alpha}{\longrightarrow}\mathrm{G}$
gives $\mathbf{M}(\mathrm{F})\overset{\mathbf{M}(\alpha)}{\longrightarrow}
\mathbf{M}(\mathrm{G})$. Applying the latter to
$M\overset{f}{\longrightarrow} M'$ yields the  commutative diagram
\begin{displaymath}
\xymatrix{ 
\mathbf{M}(\mathrm{F})\,M\ar[rr]^{\mathbf{M}(\mathrm{F})\,f}
\ar[d]_{\mathbf{M}(\alpha)_M} 
&& \mathbf{M}(\mathrm{F})\,M'\ar[d]^{\mathbf{M}(\alpha)_{M'}}\\
\mathbf{M}(\mathrm{G})\,M\ar[rr]^{\mathbf{M}(\mathrm{G})\,f} && 
\mathbf{M}(\mathrm{G})\,M'.
}
\end{displaymath}
This commutative diagram implies that 
$\{\mathbf{M}(\mathrm{F})\,f:\mathrm{F}\in
\cC(\mathtt{i},\mathtt{j})\}$ extends to a natural transformation
from $\Phi^M_{\mathtt{j}}$ to $\Phi^{M'}_{\mathtt{j}}$ and
claim \eqref{prop3.2} follows.
\end{proof}

\subsection{Connections to categories with full 
projective functors}\label{s2.3}

Denote by $\cC_{\mathtt{i}}$ the full $2$-subcategory of $\cC$
with object $\mathtt{i}$.  Restricting $\mathbf{P}_{\mathtt{i}}$
to $\mathtt{i}$ defines a (unique) principal $2$-representation
of $\cC_{\mathtt{i}}$. As $\cC$ is finitary, the identity 
$\mathbbm{1}_{\mathtt{i}}$ is indecomposable and hence so is
the projective object $P_{\mathbbm{1}_{\mathtt{i}}}$. By definition, 
for any $\mathrm{F},\mathrm{G}\in \cC_{\mathtt{i}}(\mathtt{i},\mathtt{i})$
the evaluation map 
\begin{displaymath}
\mathrm{Hom}_{\ccC_{\mathtt{i}}(\mathtt{i},\mathtt{i})}
(\mathrm{F},\mathrm{G})\to
\mathrm{Hom}_{\overline{\ccC}(\mathtt{i},\mathtt{i})}
(\mathrm{F}\circ P_{\mathbbm{1}_{\mathtt{i}}},
\mathrm{G}\circ P_{\mathbbm{1}_{\mathtt{i}}})
\end{displaymath}
is surjective (and, in fact, even bijective). Therefore the category
$\overline{\cC}(\mathtt{i},\mathtt{i})$ with the designated object 
$P_{\mathbbm{1}_{\mathtt{i}}}$ and endofunctors 
$\mathbf{P}_{\mathtt{i}}(\mathrm{F})$, 
$\mathrm{F}\in \cC_{\mathtt{i}}(\mathtt{i},\mathtt{i})$,
is a category with full projective functors in the sense of
\cite{Kh}. The notion of {\em functors naturally commuting with
projective functors} in \cite{Kh} corresponds to morphisms between 
$2$-representations of $\cC_{\mathtt{i}}$ in our language.
It might be worth pointing out that \cite{Kh} works in the
setup of bicategories (without mentioning them).

Similarly, for every $\mathtt{j}\in\cC$ and any 
$\mathrm{F},\mathrm{G}\in \cC(\mathtt{i},\mathtt{j})$
the evaluation map 
\begin{displaymath}
\mathrm{Hom}_{\ccC(\mathtt{i},\mathtt{j})}
(\mathrm{F},\mathrm{G})\to
\mathrm{Hom}_{\overline{\ccC}(\mathtt{i},\mathtt{j})}
(\mathrm{F}\circ P_{\mathbbm{1}_{\mathtt{i}}},
\mathrm{G}\circ P_{\mathbbm{1}_{\mathtt{i}}})
\end{displaymath}
is surjective (and, in fact, even bijective). 

\subsection{The regular bimodule}\label{s2.4}

For $\mathtt{i},\mathtt{j},\mathtt{k}\in\cC$ and any $1$-morphism
$\mathrm{F}\in \cC(\mathtt{k},\mathtt{i})$ the right horizontal
composition with (the identity on) $\mathrm{F}$ gives a functor
from $\overline{\cC}(\mathtt{i},\mathtt{j})$ to
$\overline{\cC}(\mathtt{k},\mathtt{j})$. For any $1$-morphisms
$\mathrm{F},\mathrm{G}\in \cC(\mathtt{k},\mathtt{i})$ and
a $2$-morphism $\alpha:\mathrm{F}\to \mathrm{G}$ the right 
horizontal composition with $\alpha$ gives a natural transformation
between the corresponding functors.
This turns $\overline{\cC}(\cdot,\cdot)$ into a $2$-bimodule  over
$\cC$. This bimodule is called the {\em regular bimodule}.

\subsection{The abelian envelope of $\cC$}\label{s2.5}

Because of the previous subsection, it is natural to expect that
one could turn $\overline{\cC}$ into a $2$-category with the same 
set of objects as $\cC$. Unfortunately, we do not know how to do
this as it seems that $\overline{\cC}$ contains ``too many''
objects (and hence only has the natural structure of a bicategory). 
Instead, we define a biequivalent $2$-category $\hat{\cC}$ as follows: 
Objects of $\hat{\cC}$ are objects of $\cC$. 
To define $1$-morphisms of $\hat{\cC}$ consider the
regular $2$-bimodule $\overline{\cC}(\cdot,\cdot)$ over
$\cC$ just as a left $2$-representation. Let $\cR$ be the 
$2$-category with same objects as $\cC$ and such that for
$\mathtt{i},\mathtt{j}\in\cC$ the category $\cR(\mathtt{i},\mathtt{j})$
is defined as the category of all functors from
$\displaystyle \bigoplus_{\mathtt{k}\in\overline{\ccC}}
\overline{\cC}(\mathtt{k},\mathtt{i})$ to
$\displaystyle\bigoplus_{\mathtt{k}\in\overline{\ccC}}
\overline{\cC}(\mathtt{k},\mathtt{j})$, where morphisms are
all natural transformations of functors. We are going to define
$\hat{\cC}$ as a $2$-subcategory of $\cR$.

The regular bimodule $2$-representation of $\cC$ is a $2$-functor
from $\cC$ to $\cR$ (which is the identity on objects). As usual, for 
every $\mathtt{i},\mathtt{j}\in\cC$ and any
$\mathrm{F}\in\cC(\mathtt{i},\mathtt{j})$ we will denote the image
of $\mathrm{F}$ under this $2$-functor also by $\mathrm{F}$.
We define $1$-morphisms in $\hat{\cC}(\mathtt{i},\mathtt{j})$
as functors in $\cR(\mathtt{i},\mathtt{j})$ 
of the form $\mathrm{Coker}(\alpha)$, where 
$\alpha$ is a $2$-morphism from $\mathrm{F}$ to $\mathrm{G}$ for some
$\mathrm{F},\mathrm{G}\in \cC(\mathtt{i},\mathtt{j})$.
We define $2$-morphisms in $\hat{\cC}(\mathtt{i},\mathtt{j})$
as natural transformations between the corresponding cokernel
functors coming from commutative diagrams of the following form, 
where all solid arrows are $2$-morphisms in $\cC$:
\begin{displaymath}
\xymatrix{ 
\mathrm{F}\ar[rr]^{\alpha}\ar[d]_{\xi'}&&\mathrm{G}\ar[d]^{\xi}
\ar@{.>>}[rr]^{\mathrm{proj}}&&\mathrm{Coker}(\alpha)\ar@{.>}[d]\\
\mathrm{F}'\ar[rr]^{\alpha'}&&\mathrm{G}'
\ar@{.>>}[rr]^{\mathrm{proj}}&&\mathrm{Coker}(\alpha')\\
}
\end{displaymath}

\begin{lemma}\label{lll123}
\begin{enumerate}[$($a$)$]
\item \label{lll123.1} $1$-morphisms in $\hat{\cC}$ are closed with 
respect to  the usual composition of functors in $\cR$.
\item \label{lll123.2} $2$-morphisms  in $\hat{\cC}$ are closed with
respect to both horizontal and vertical compositions in $\cR$.
\end{enumerate}
\end{lemma}

\begin{proof}
Let $\mathtt{i},\mathtt{j},\mathtt{k}\in \cC$, 
$\mathrm{F},\mathrm{G}\in {\cC}(\mathtt{i},\mathtt{j})$,
$\mathrm{F}',\mathrm{G}'\in {\cC}(\mathtt{j},\mathtt{k})$
and $\xymatrix{\mathrm{F}\ar[r]^{\alpha}&\mathrm{G}}$,
$\xymatrix{\mathrm{F}'\ar[r]^{\alpha'}&\mathrm{G}'}$ be some
$2$-morphisms. Then the interchange law
for the $2$-category $\cC$ yields that the following diagram is commutative:
\begin{displaymath}
\xymatrix{ 
\mathrm{F}'\circ \mathrm{F}
\ar[rr]^{\mathrm{F}'(\alpha)}
\ar[d]_{\alpha'_{\mathrm{F}}}  && 
\mathrm{F}'\circ \mathrm{G}
\ar[d]^{\alpha'_{\mathrm{G}}} \\
\mathrm{G}'\circ \mathrm{F}
\ar[rr]^{\mathrm{G}'(\alpha)}  && 
\mathrm{G}'\circ \mathrm{G} 
}
\end{displaymath}
This means that 
\begin{displaymath}
\mathrm{Coker}(\alpha')\circ \mathrm{Coker}(\alpha)=
\mathrm{Coker}((\alpha'_{\mathrm{G}},\mathrm{G}'(\alpha))),
\end{displaymath}
where $(\alpha'_{\mathrm{G}},\mathrm{G}'(\alpha))$ is given
by the following diagram:
\begin{displaymath}
\xymatrix{(\mathrm{F}'\circ \mathrm{G})\oplus
(\mathrm{G}'\circ \mathrm{F})
\ar[rrrr]^{(\alpha'_{\mathrm{G}},\mathrm{G}'(\alpha))}&&&&
\mathrm{G}'\circ \mathrm{G}}.
\end{displaymath}
This implies claim \eqref{lll123.1}.

That $2$-morphisms are closed with respect to vertical composition
follows directly from the definitions. To see that 
$2$-morphisms are closed with respect to horizontal composition,
consider the following two commutative diagrams in $\cC$:
\begin{displaymath}
\xymatrix{
\mathrm{F}_1\ar[rr]^{\alpha}\ar[d]_{\xi_1}&&\mathrm{G}_1
\ar[d]^{\eta_1}\\
\mathrm{F}'_1\ar[rr]^{\alpha'}&&\mathrm{G}'_1
}\qquad \hbox{and} \qquad
\xymatrix{
\mathrm{F}_2\ar[rr]^{\beta}\ar[d]_{\xi_2}&&\mathrm{G}_2
\ar[d]^{\eta_2}\\
\mathrm{F}'_2\ar[rr]^{\beta'}&&\mathrm{G}'_2
}
\end{displaymath}
These diagrams induce $2$-morphisms between the corresponding
cokernels. The horizontal composition of these two morphisms is
induced by the following commutative diagram:
\begin{displaymath}
\xymatrix{
\mathrm{F}_1\circ \mathrm{G}_2 \oplus \mathrm{G}_1 \circ \mathrm{F}_2 \ar[rrrr]^{(\alpha_{\mathrm{G}_2},\mathrm{G}_1(\beta))}
\ar[dd]_{\left( \begin{array}{cc}                                                                                                                       \xi_1 \circ_0 \eta_2 & 0\\0 & \eta_1 \circ_0 \xi_2                                                                                                                        \end{array}\right)}&&&&\mathrm{G}_1 \circ\mathrm{G}_2
\ar[dd]^{\eta_1 \circ_0 \eta_2}\\ \\
\mathrm{F}'_1\circ \mathrm{G}'_2\oplus \mathrm{G}'_1\circ \mathrm{F}'_2\ar[rrrr]^{(\alpha'_{\mathrm{G}'_2}, 
\mathrm{G}'_1(\beta'))}&&&&\mathrm{G}'_1\circ \mathrm{G}'_2.
}
\end{displaymath}

This proves claim \eqref{lll123.2} and completes the proof.
\end{proof}

From Lemma~\ref{lll123} it follows that $\hat{\cC}$ is a $2$-subcategory
of $\cR$. From the construction it also follows that for any
$\mathtt{i},\mathtt{j}\in\cC$ the categories
$\overline{\cC}(\mathtt{i},\mathtt{j})$ and
$\hat{\cC}(\mathtt{i},\mathtt{j})$ are equivalent.
Furthermore, directly from the definitions we have:

\begin{lemma}\label{lem111}
There is a unique full and faithful $2$-functor 
$\mathfrak{i}:\cC\to \hat{\cC}$ such that 
for any $\mathtt{i},\mathtt{j}\in \cC$,
$\mathrm{F},\mathrm{G}\in  \cC(\mathtt{i},\mathtt{j})$
and $\alpha:\mathrm{F}\to\mathrm{G}$ we have 
$\mathfrak{i}(\mathrm{F})=
\mathrm{Coker}(\xymatrix{0\ar[r]^{0}&\mathrm{F}})$ and 
$\mathfrak{i}(\alpha)$ is induces by
\begin{displaymath}
\xymatrix{0\ar[r]^{0}\ar[d]_{0}&\mathrm{F}\ar[d]^{\alpha}\\
0\ar[r]^{0}&\mathrm{G}}.
\end{displaymath}
\end{lemma}

As usual, the $2$-functor $\mathfrak{i}$ induces the restriction 
$2$-functor $\hat{\mathfrak{i}}:\hat{\cC}\text{-}
\mathfrak{mod}\to\cC\text{-}\mathfrak{mod}$. For the opposite
direction we have:

\begin{theorem}\label{thm112}
Every $2$-representation of $\cC$ extends to a $2$-representation 
of $\hat{\cC}$.
\end{theorem}

\begin{proof} 
Let $\mathbf{M}\in \cC\text{-}\mathfrak{mod}$. Abusing notation we
will denote the extension of $\mathbf{M}$ to a $2$-representation 
of $\hat{\cC}$ also by $\mathbf{M}$. Let $\mathtt{i},\mathtt{j}\in \cC$,
$\mathrm{F},\mathrm{G}\in  \cC(\mathtt{i},\mathtt{j})$
and $\alpha:\mathrm{F}\to\mathrm{G}$. Then for
$\mathrm{Coker}(\alpha)\in \hat{\cC}(\mathtt{i},\mathtt{j})$
we define $\mathbf{M}(\mathrm{Coker}(\alpha))$ as
$\mathrm{Coker}(\mathbf{M}(\alpha))$.

To define $\mathbf{M}$ on $2$-morphisms in $\hat{\cC}$,
let $\mathrm{F}',\mathrm{G}'\in  \cC(\mathtt{i},\mathtt{j})$,
$\alpha':\mathrm{F}'\to\mathrm{G}'$, $\beta:\mathrm{F}\to\mathrm{F}'$
and $\beta':\mathrm{G}\to\mathrm{G}'$ are such that the diagram
\begin{displaymath}
\Gamma:=\xymatrix{ 
\mathrm{F}\ar[rr]^{\alpha}\ar[d]_{\beta}&&\mathrm{G}\ar[d]^{\beta'}\\ 
\mathrm{F}'\ar[rr]^{\alpha'}&&\mathrm{G}'\\ 
}
\end{displaymath}
is commutative. Then a typical $2$-morphism $\gamma$ 
in $\hat{\cC}$ is induced by $\Gamma$. Applying $\mathbf{M}$ induces 
the commutative solid part of the following diagram:
\begin{displaymath}
\xymatrix{ 
\mathbf{M}(\mathrm{F})\ar[rr]^{\mathbf{M}(\alpha)}
\ar[d]_{\mathbf{M}(\beta)}&&
\mathbf{M}(\mathrm{G})\ar[d]^{\mathbf{M}(\beta')}
\ar@{-->}[rr]^{\mathrm{proj}}&&
\mathrm{Coker}(\mathbf{M}(\alpha))\ar@{-->>}[d]^{\xi}\\ 
\mathbf{M}(\mathrm{F}')\ar[rr]^{\mathbf{M}(\alpha')}&&
\mathbf{M}(\mathrm{G}')\ar@{-->>}[rr]^{\mathrm{proj}}&&
\mathrm{Coker}(\mathbf{M}(\alpha'))\\ 
}
\end{displaymath}
Because of the commutativity of the solid part, the diagram extends 
uniquely to a commutative diagram by the dashed arrows as shown above.
Directly from the construction it follows that $\mathbf{M}$ becomes 
a $2$-representation of $\hat{\cC}$.
\end{proof}

Because of Theorem~\ref{thm112} it is natural to call $\hat{\cC}$
the {\em abelian envelope} of $\cC$. In what follows we will always
view $2$-representation of ${\cC}$ as $2$-representation of $\hat{\cC}$
via the construction given by Theorem~\ref{thm112}.

\section{Cells and cell $2$-representations of fiat categories}\label{s3}

From now on we assume that $\cC$ is a fiat category.

\subsection{Orders and cells}\label{s3.1}

Set $\mathcal{C}=\cup_{\mathtt{i},\mathtt{j}}
\mathcal{C}_{\mathtt{i},\mathtt{j}}$.
Let $\mathtt{i},\mathtt{j},\mathtt{k},\mathtt{l}\in\cC$,
$\mathrm{F}\in\mathcal{C}_{\mathtt{i},\mathtt{j}}$ and
$\mathrm{G}\in\mathcal{C}_{\mathtt{k},\mathtt{l}}$. 
We will write $\mathrm{F}\leq_R \mathrm{G}$ provided that there 
exists $\mathrm{H}\in\cC(\mathtt{j},\mathtt{l})$ such that $\mathrm{G}$
occurs as a direct summand of $\mathrm{H}\circ\mathrm{F}$
(note that this is possible only if $\mathtt{i}=\mathtt{k}$).
Similarly, we will write $\mathrm{F}\leq_L \mathrm{G}$ provided that there 
exists $\mathrm{H}\in\cC(\mathtt{k},\mathtt{i})$ such that $\mathrm{G}$
occurs as a direct summand of $\mathrm{F}\circ\mathrm{H}$
(note that this is possible only if $\mathtt{j}=\mathtt{l}$).
Finally, we will write $\mathrm{F}\leq_{LR} \mathrm{G}$ provided that there 
exists $\mathrm{H}_1\in\cC(\mathtt{k},\mathtt{i})$ 
and $\mathrm{H}_2\in\cC(\mathtt{j},\mathtt{l})$ such that $\mathrm{G}$
occurs as a direct summand of $\mathrm{H}_2\circ\mathrm{F}\circ\mathrm{H}_1$.
The relations $\leq_{L}$, $\leq_{R}$ and $\leq_{LR}$ are partial
preorders on $\mathcal{C}$. The map $\mathrm{F}\mapsto \mathrm{F}^*$
preserves $\leq_{LR}$ and swaps $\leq_{L}$ and $\leq_{R}$.

For $\mathrm{F}\in \mathcal{C}$ the set of all 
$\mathrm{G}\in \mathcal{C}$ such that $\mathrm{F}\leq_R \mathrm{G}$
and $\mathrm{G}\leq_R \mathrm{F}$ will be called the {\em right
cell} of $\mathrm{F}$ and denoted by $\mathcal{R}_{\mathrm{F}}$.
The {\em left cell}  $\mathcal{L}_{\mathrm{F}}$ 
and the {\em two-sided cell}  $\mathcal{LR}_{\mathrm{F}}$ 
are defined analogously. We will write $\mathrm{F}\sim_R \mathrm{G}$
provided that $\mathrm{G}\in \mathcal{R}_{\mathrm{F}}$ 
and define $\sim_L$ and $\sim_{LR}$
analogously. These are equivalence relations on $\mathcal{C}$.
If $\mathrm{F}\leq_L \mathrm{G}$ and $\mathrm{F}\not\sim_L \mathrm{G}$,
then we will write $\mathrm{F}<_L \mathrm{G}$ and similarly
for $<_R$ and $<_{LR}$.

\begin{example}\label{exm4}
{\rm  
The $2$-category $\cS_2$ from Example~\ref{exm1} has two right cells
$\{\mathbbm{1}_{\mathtt{i}}\}$ and $\{\mathrm{F}\}$, which are also left cells
and thus two-sided cells as well.
}
\end{example}

\subsection{Annihilators and filtrations}\label{s3.2}

Let $\mathbf{M}$ be a $2$-representation of $\cC$. For any 
$\mathtt{i}\in\cC$ and any $M\in\mathbf{M}(\mathtt{i})$ consider 
the {\em annihilator} $\mathrm{Ann}_{\mathcal{C}}(M):=
\{\mathrm{F}\in\mathcal{C}:\mathrm{F}\,M=0\}$ of $M$.
The set $\mathrm{Ann}_{\mathcal{C}}(M)$ is a coideal with respect 
to $\leq_{R}$ in the sense that $\mathrm{F}\in 
\mathrm{Ann}_{\mathcal{C}}(M)$ and $\mathrm{F}\leq_{R}\mathrm{G}$
implies $\mathrm{G}\in \mathrm{Ann}_{\mathcal{C}}(M)$. The 
{\em annihilator} $\displaystyle\mathrm{Ann}_{\mathcal{C}}(\mathbf{M}):=
\bigcap_{M} \mathrm{Ann}_{\mathcal{C}}(M)$ of $\mathbf{M}$
is a coideal with respect to $\leq_{LR}$.

Let $\mathcal{I}$ be a coideal in $\mathcal{C}$ with respect to $\leq_{LR}$.
For every $\mathtt{i}\in\cC$ denote by $\mathbf{M}_{\mathcal{I}}(\mathtt{i})$
the Serre subcategory of $\mathbf{M}(\mathtt{i})$ generated by all
simple modules $L$ such that 
$\mathcal{I}\subset \mathrm{Ann}_{\mathcal{C}}(L)$.

\begin{lemma}\label{lem5}
By restriction, $\mathbf{M}_{\mathcal{I}}$ is a $2$-representation of $\cC$.
\end{lemma}

\begin{proof}
We need to check that $\mathbf{M}_{\mathcal{I}}$ is stable under the 
action of  elements from $\mathcal{C}$. If $L$ is a simple module in 
$\mathbf{M}_{\mathcal{I}}(\mathtt{i})$ and $\mathrm{F}\in 
\mathcal{C}_{\mathtt{i},\mathtt{j}}$, then for any 
$\mathrm{G}\in \mathcal{I}$ the $1$-morphism $\mathrm{G}\circ \mathrm{F}$
is either zero or decomposes into a direct sum of $1$-morphisms in 
$\mathcal{I}$ (as $\mathcal{I}$ is a coideal with respect to $\leq_{LR}$).
This implies $\mathrm{G}\circ \mathrm{F}\, L=0$. Exactness
of $\mathrm{G}$ implies that $\mathrm{G}\, K=0$ for any simple
subquotient $K$ of $\mathrm{F}\, L$. The claim follows.
\end{proof}

Assume that for any $\mathtt{i}\in\cC$ we fix some Serre subcategory
$\mathbf{N}(\mathtt{i})$ in $\mathbf{M}(\mathtt{i})$ such that
for any $\mathtt{j}\in\cC$ and any $\mathrm{F}\in \cC(\mathtt{i},\mathtt{j})$ 
we have $\mathrm{F}\,\mathbf{N}(\mathtt{i})\subset \mathbf{N}(\mathtt{j})$.
Then $\mathbf{N}(\mathtt{i})$ is a $2$-representation of $\cC$
by restriction. It will be called a {\em Serre $2$-subrepresentation}  
of $\mathbf{M}$. For example, the $2$-representation 
$\mathbf{M}_{\mathcal{I}}$ constructed in Lemma~\ref{lem5} is a Serre 
$2$-subrepresentation of $\mathbf{M}$.

\begin{proposition}\label{prop6}
\begin{enumerate}[$($a$)$]
\item\label{prop6.1} For any coideal $\mathcal{I}$ in $\mathcal{C}$ with 
respect to $\leq_{LR}$ we have 
\begin{displaymath}
\mathbf{M}_{\mathcal{I}}= \mathbf{M}_{\mathrm{Ann}_{\mathcal{C}}(\mathbf{M}_{\mathcal{I}})}.
\end{displaymath}
\item\label{prop6.2} For any Serre $2$-subrepresentation $\mathbf{N}$
of $\mathbf{M}$ we have 
\begin{displaymath}
\mathrm{Ann}_{\mathcal{C}}(\mathbf{N}) = 
\mathrm{Ann}_{\mathcal{C}}(\mathbf{M}_{\mathrm{Ann}_{\mathcal{C}}(\mathbf{N})}).
\end{displaymath}
\end{enumerate}
\end{proposition}

\begin{proof}
We prove claim \eqref{prop6.2}. Claim \eqref{prop6.1} is proved
similarly. By definition, for every 
$\mathtt{i}\in \cC$ we have $\mathbf{N}(\mathtt{i})\subset
\mathbf{M}_{\mathrm{Ann}_{\mathcal{C}}(\mathbf{N})}(\mathtt{i})$.
This implies $\mathrm{Ann}_{\mathcal{C}}(
\mathbf{M}_{\mathrm{Ann}_{\mathcal{C}}(\mathbf{N})})\subset 
\mathrm{Ann}_{\mathcal{C}}(\mathbf{N})$. On the other hand,
by definition $\mathrm{Ann}_{\mathcal{C}}(\mathbf{N})$
annihilates $\mathbf{M}_{\mathrm{Ann}_{\mathcal{C}}(\mathbf{N})}$,
so $\mathrm{Ann}_{\mathcal{C}}(\mathbf{N})\subset
\mathrm{Ann}_{\mathcal{C}}(\mathbf{M}_{\mathrm{Ann}_{\mathcal{C}}(\mathbf{N})})$. This completes the proof.
\end{proof}

Proposition~\ref{prop6} says that ${\mathcal{I}}\mapsto 
\mathbf{M}_{\mathcal{I}}$
and $\mathbf{N}\mapsto \mathrm{Ann}_{\mathcal{C}}(\mathbf{N})$
is a Galois correspondence between the partially ordered set of
coideals in $\mathcal{C}$ with respect to $\leq_{LR}$ and 
the partially ordered set of Serre $2$-subrepresentations of 
$\mathbf{M}$ with respect to inclusions.

\subsection{Annihilators in principal $2$-representations}\label{s3.3}

Let $\mathtt{i}\in\cC$. By construction, for $\mathtt{j}\in\cC$ 
isomorphism classes of simple modules in $\mathbf{P}_{\mathtt{i}}(\mathtt{j})$
are indexed by $\mathcal{C}_{\mathtt{i},\mathtt{j}}$. For
$\mathrm{F}\in \mathcal{C}_{\mathtt{i},\mathtt{j}}$ we denote 
by $L_{\mathrm{F}}$ the unique simple quotient of $P_{\mathrm{F}}$.

\begin{lemma}\label{lem11}
For $\mathrm{F},\mathrm{G}\in \mathcal{C}$ the inequality
$\mathrm{F}\,L_{\mathrm{G}}\neq 0$ is equivalent to 
$\mathrm{F}^*\leq_L\mathrm{G}$.
\end{lemma}

\begin{proof}
Without loss of generality we may assume 
$\mathrm{G}\in \mathcal{C}_{\mathtt{i},\mathtt{j}}$
and $\mathrm{F}\in \mathcal{C}_{\mathtt{j},\mathtt{k}}$.
Then $\mathrm{F}\,L_{\mathrm{G}}\neq 0$ if and only if there is 
$\mathrm{H}\in\mathcal{C}_{\mathtt{i},\mathtt{k}}$ such that 
$\mathrm{Hom}_{\overline{\ccC}(\mathtt{i},\mathtt{k})}
(P_{\mathrm{H}},\mathrm{F}\,L_{\mathrm{G}})\neq 0$. Using
$P_{\mathrm{H}}=\mathrm{H}\, P_{\mathbbm{1}_{\mathtt{i}}}$
and adjunction we obtain
\begin{displaymath}
0\neq \mathrm{Hom}_{\overline{\ccC}(\mathtt{i},\mathtt{k})}
(P_{\mathrm{H}},\mathrm{F}\,L_{\mathrm{G}})=
\mathrm{Hom}_{\overline{\ccC}(\mathtt{i},\mathtt{j})}
(\mathrm{F}^*\circ\mathrm{H}\, P_{\mathbbm{1}_{\mathtt{i}}},L_{\mathrm{G}}).
\end{displaymath}
This inequality is equivalent to the claim that
$P_{\mathrm{G}}=\mathrm{G}\, P_{\mathbbm{1}_{\mathtt{i}}}$ 
is a direct summand of $\mathrm{F}^*\circ\mathrm{H}\, 
P_{\mathbbm{1}_{\mathtt{i}}}$, that is $\mathrm{G}$ is a direct summand
of $\mathrm{F}^*\circ\mathrm{H}$. The claim follows.
\end{proof}

\begin{lemma}\label{lem12}
\begin{enumerate}[$($a$)$]
\item\label{lem12.1} For $\mathrm{F},\mathrm{G},\mathrm{H}\in \mathcal{C}$ 
the inequality $[\mathrm{F}\,L_{\mathrm{G}}:L_{\mathrm{H}}]\neq 0$ 
implies $\mathrm{H}\leq_R\mathrm{G}$.
\item\label{lem12.2} For $\mathrm{G},\mathrm{H}\in \mathcal{C}$ 
such that $\mathrm{H}\leq_R\mathrm{G}$ there exists 
$\mathrm{F}\in \mathcal{C}$ such that 
$[\mathrm{F}\,L_{\mathrm{G}}:L_{\mathrm{H}}]\neq 0$.
\end{enumerate}
\end{lemma}

\begin{proof}
Without loss of generality we may assume 
\begin{equation}\label{eq111}
\mathrm{G}\in \mathcal{C}_{\mathtt{i},\mathtt{j}},\quad
\mathrm{F}\in \mathcal{C}_{\mathtt{j},\mathtt{k}}\quad\text{ and } \quad
\mathrm{H}\in \mathcal{C}_{\mathtt{i},\mathtt{k}}.
\end{equation}
Then $[\mathrm{F}\,L_{\mathrm{G}}:L_{\mathrm{H}}]\neq 0$
is equivalent to 
$\mathrm{Hom}_{\overline{\ccC}(\mathtt{i},\mathtt{k})}
(P_{\mathrm{H}},\mathrm{F}\,L_{\mathrm{G}})\neq 0$.
Similarly to Lemma~\ref{lem11} we obtain that $\mathrm{G}$
must be a direct summand of $\mathrm{F}^*\circ\mathrm{H}$.
This means that $\mathrm{H}\leq_R\mathrm{G}$, proving \eqref{lem12.1}.

To prove \eqref{lem12.2} we note that $\mathrm{H}\leq_R\mathrm{G}$
implies existence of $\mathrm{F}\in \mathcal{C}$ such that
$\mathrm{G}$ is a direct summand of $\mathrm{F}^*\circ \mathrm{H}$.
We may assume that $\mathrm{F},\mathrm{G}$ and $\mathrm{H}$ are as
in \eqref{eq111}. Then, by adjunction, we have
\begin{displaymath}
0\neq \mathrm{Hom}_{\overline{\ccC}(\mathtt{i},\mathtt{k})}
(\mathrm{F}^*\,P_{\mathrm{H}}, L_{\mathrm{G}})=
\mathrm{Hom}_{\overline{\ccC}(\mathtt{i},\mathtt{j})}
(P_{\mathrm{H}},\mathrm{F}\,L_{\mathrm{G}}),
\end{displaymath}
which means that $[\mathrm{F}\,L_{\mathrm{G}}:L_{\mathrm{H}}]\neq 0$.
This completes the proof.
\end{proof}

\begin{corollary}\label{cor121}
Let $\mathrm{F},\mathrm{G},\mathrm{H}\in \mathcal{C}$.
If $L_{\mathrm{F}}$ occurs in the top or in the socle of
$\mathrm{H}\,L_{\mathrm{G}}$, then $\mathrm{F}\in \mathcal{R}_{\mathrm{G}}$.
\end{corollary}

\begin{proof}
We prove the claim in the case when $L_{\mathrm{F}}$ occurs in the 
top of $\mathrm{H}\,L_{\mathrm{G}}$, the other case being analogous.
As $[\mathrm{H}\,L_{\mathrm{G}}:L_{\mathrm{F}}]\neq 0$, we have
$\mathrm{F}\leq_R \mathrm{G}$ by Lemma~\ref{lem12}. On the other hand,
by adjunction, $L_{\mathrm{G}}$ occurs in the 
socle of $\mathrm{H}^*\,L_{\mathrm{F}}$. Hence $[\mathrm{H}^*\,L_{\mathrm{F}}:L_{\mathrm{G}}]\neq 0$ and thus we have
$\mathrm{G}\leq_R \mathrm{F}$ by Lemma~\ref{lem12}. The claim follows.
\end{proof}

\begin{lemma}\label{lem122}
For any $\mathrm{F}\in \mathcal{C}_{\mathtt{i},\mathtt{j}}$ there is a unique 
(up to scalar) nontrivial homomorphism from $P_{\mathbbm{1}_{\mathtt{i}}}$
to $\mathrm{F}^*L_{\mathrm{F}}$. In particular, 
$\mathrm{F}^*L_{\mathrm{F}}\neq 0$.
\end{lemma}

\begin{proof}
Adjunction yields
\begin{displaymath}
\mathrm{Hom}_{\overline{\ccC}(\mathtt{i},\mathtt{i})}
(P_{\mathbbm{1}_{\mathtt{i}}},\mathrm{F}^*\,L_{\mathrm{F}})=
\mathrm{Hom}_{\overline{\ccC}(\mathtt{i},\mathtt{j})}
(P_{\mathrm{F}},L_{\mathrm{F}})\cong \Bbbk 
\end{displaymath}
and the claim follows.
\end{proof}

\subsection{Serre $2$-subrepresentations of 
$\mathbf{P}_{\mathtt{i}}$}\label{s3.35}

Let $\mathcal{I}$ be an ideal in $\mathcal{C}$ 
with respect to $\leq_{R}$, i.e. $\mathrm{F}\in \mathcal{I}$ 
and $\mathrm{F}\geq_{R}\mathrm{G}$ implies $\mathrm{G}\in \mathcal{I}$. 
For $\mathtt{i},\mathtt{j}\in\cC$ 
define $\mathbf{P}_{\mathtt{i}}^{\mathcal{I}}
(\mathtt{j})$ as the Serre subcategory of 
$\mathbf{P}_{\mathtt{i}}(\mathtt{j})$ generated by 
$L_{\mathrm{F}}$ for $\mathrm{F}\in
\mathcal{C}_{\mathtt{i},\mathtt{j}}\cap\mathcal{I}$. Then from
Lemma~\ref{lem12} it follows that $\mathbf{P}_{\mathtt{i}}^{\mathcal{I}}$
is a Serre $2$-subrepresentation of $\mathbf{P}_{\mathtt{i}}$ and that
every Serre $2$-subrepresentation of $\mathbf{P}_{\mathtt{i}}$ arises
in this way. For $\mathrm{F}\in\mathcal{I}\cap
\mathcal{C}_{\mathtt{i},\mathtt{j}}$ we denote by
$P^{\mathcal{I}}_{\mathrm{F}}$ the maximal quotient of
$P_{\mathrm{F}}$ in $\mathbf{P}_{\mathtt{i}}^{\mathcal{I}}(\mathtt{j})$.
The module $P^{\mathcal{I}}_{\mathrm{F}}$ is a projective cover 
of $L_{\mathrm{F}}$ in $\mathbf{P}_{\mathtt{i}}^{\mathcal{I}}(\mathtt{j})$.
Since $\mathbf{P}_{\mathtt{i}}^{\mathcal{I}}$ is a 
$2$-subrepresentation of $\mathbf{P}_{\mathtt{i}}$,
for $\mathrm{F}\in\mathcal{C}$ we have
\begin{displaymath}
\mathrm{F}\,P^{\mathcal{I}}_{\mathbbm{1}_{\mathtt{i}}}=
\begin{cases}
P^{\mathcal{I}}_{\mathrm{F}}, & \mathrm{F}\in\mathcal{I};\\
0, & \text{otherwise}.
\end{cases}
\end{displaymath}
From the definition we have that $2$-morphisms in $\cC$ surject onto 
homomorphisms between the various $P^{\mathcal{I}}_{\mathrm{F}}$.
The natural inclusion $\mathfrak{i}^{\mathcal{I}}:
\mathbf{P}_{\mathtt{i}}^{\mathcal{I}}\to\mathbf{P}_{\mathtt{i}}$ is a 
morphism of $2$-representations, given by the collection of exact inclusions
$\mathfrak{i}^{\mathcal{I}}_{\mathtt{j}}:\mathbf{P}_{\mathtt{i}}^{\mathcal{I}}
(\mathtt{j})\to\mathbf{P}_{\mathtt{i}}(\mathtt{j})$. 

Note that $\mathcal{C}\setminus \mathcal{I}$ is
a coideal in $\mathcal{C}$ with respect to $\leq_{R}$.
Hence for any $2$-representation $\mathbf{M}$  of $\cC$ 
we have the corresponding Serre $2$-subrepresentation 
$\mathbf{M}_{\mathcal{C}\setminus \mathcal{I}}$ of $\mathbf{M}$.

\begin{proposition}[Universal property of $\mathbf{P}_{\mathtt{i}}^{\mathcal{I}}$]
\label{prop26}
Let $\mathbf{M}$ a $2$-representation of $\cC$.
\begin{enumerate}[$($a$)$]
\item\label{prop26.1} For any morphism 
$\Phi:\mathbf{P}_{\mathtt{i}}^{\mathcal{I}}\to \mathbf{M}$ we have
$\Phi(P^{\mathcal{I}}_{\mathbbm{1}_{\mathtt{i}}})\in 
\mathbf{M}_{\mathcal{C}\setminus \mathcal{I}}(\mathtt{i})$.
\item\label{prop26.2}
Let $M\in \mathbf{M}_{\mathcal{C}\setminus \mathcal{I}}(\mathtt{i})$. 
For $\mathtt{j}\in \cC$ let $\Phi^M_{\mathtt{j}}:
\mathbf{P}_{\mathtt{i}}^{\mathcal{I}}(\mathtt{j})\to 
\mathbf{M}(\mathtt{j})$ be the unique right exact functor such that 
for any $\mathrm{F}\in\cC(\mathtt{i},\mathtt{j})$ we have
\begin{displaymath}
\Phi^M_{\mathtt{j}}: P^{\mathcal{I}}_{\mathrm{F}}
\mapsto \mathbf{M}(\mathrm{F})\,M.
\end{displaymath}
Then $\Phi^M=(\Phi^M_{\mathtt{j}})_{\mathtt{j}\in\ccC}:
\mathbf{P}_{\mathtt{i}}\to \mathbf{M}$ is the unique morphism
sending $P^{\mathcal{I}}_{\mathrm{F}}$ to $M$.
\item\label{prop26.3} 
The correspondence  $M\mapsto \Phi^M$ is functorial.
\end{enumerate}
\end{proposition}

\begin{proof}
Claim \eqref{prop26.1} follows from the fact that
$\mathcal{C}\setminus \mathcal{I}\subset
\mathrm{Ann}_{\mathcal{C}}(P^{\mathcal{I}}_{\mathbbm{1}_{\mathtt{i}}})$.
Mutatis mutandis, the rest is Proposition~\ref{prop3}.
\end{proof}

\subsection{Right cell $2$-representations}\label{s3.4}

Fix $\mathtt{i}\in\cC$. Let $\mathcal{R}$ be a right cell in 
$\mathcal{C}$ such that $\mathcal{R}\cap 
\mathcal{C}_{\mathtt{i},\mathtt{j}}\neq \varnothing$ for
some $\mathtt{j}\in\cC$.

\begin{proposition}\label{lem7}
\begin{enumerate}[$($a$)$]
\item\label{lem7.01}
There is a unique submodule $K=K_{\mathcal{R}}$ of 
$P_{\mathbbm{1}_{\mathtt{i}}}$ which has the following
properties:
\begin{enumerate}[$($i$)$]
\item\label{lem7.1} Every simple subquotient of 
$P_{\mathbbm{1}_{\mathtt{i}}}/K$ is annihilated by any
$\mathrm{F}\in\mathcal{R}$.
\item\label{lem7.2} The module $K$ has simple top 
$L_{\mathrm{G}_{\mathcal{R}}}$ for some 
$\mathrm{G}_{\mathcal{R}}\in\mathcal{C}$ and
$\mathrm{F}\, L_{\mathrm{G}_{\mathcal{R}}}\neq 0$
for any $\mathrm{F}\in\mathcal{R}$.
\end{enumerate}
\item\label{lem7.02} For any $\mathrm{F}\in\mathcal{R}$
the module $\mathrm{F}\, L_{\mathrm{G}_{\mathcal{R}}}$
has simple top $L_{\mathrm{F}}$.
\item\label{lem7.03} We have $\mathrm{G}_{\mathcal{R}}\in\mathcal{R}$.
\item\label{lem7.04} For any $\mathrm{F}\in\mathcal{R}$
we have $\mathrm{F}^*\leq_L \mathrm{G}_{\mathcal{R}}$
and $\mathrm{F}\leq_R \mathrm{G}_{\mathcal{R}}^*$.
\item\label{lem7.05} We have $\mathrm{G}^*_{\mathcal{R}}\in 
\mathcal{R}$.
\end{enumerate}
\end{proposition}

\begin{proof}
Let $\mathrm{F}\in\mathcal{R}$. Let further 
$\mathtt{j}\in\cC$ be such that $\mathrm{F}\in\mathcal{R}\cap
\mathcal{C}_{\mathtt{i},\mathtt{j}}$. Then
the module $\mathrm{F}\, P_{\mathbbm{1}_{\mathtt{i}}}$ 
is a nonzero indecomposable projective
in $\overline{\cC(\mathtt{i},\mathtt{j})}$.
Hence $\mathrm{F}$ does not annihilate 
$P_{\mathbbm{1}_{\mathtt{i}}}$ and thus there is at
least one simple subquotient of 
$P_{\mathbbm{1}_{\mathtt{i}}}$ which is not annihilated by
$\mathrm{F}$. Let $K$ be the minimal submodule of 
$P_{\mathbbm{1}_{\mathtt{i}}}$ such that 
every simple subquotient of 
$P_{\mathbbm{1}_{\mathtt{i}}}/K$ is annihilated by $\mathrm{F}$.
As $\mathrm{Ann}_{\mathcal{C}}(P_{\mathbbm{1}_{\mathtt{i}}}/K)$
is a coideal with respect to $\leq_R$, the module
$P_{\mathbbm{1}_{\mathtt{i}}}/K$ is annihilated 
by every $\mathrm{G}\in\mathcal{R}$. Similarly, we have that
for any simple subquotient $L$ in the top of $K$ and for any
$\mathrm{G}\in\mathcal{R}$ we have $\mathrm{G}\, L\neq 0$. 
This implies that $K$ 
does not depend on the choice of $\mathrm{F}\in\mathcal{R}$.
Then \eqref{lem7.1} is satisfied and to complete the proof of
\eqref{lem7.01} we only have to show that $K$ has simple top.

Applying $\mathrm{F}$ to the exact sequence
$K\hookrightarrow P_{\mathbbm{1}_{\mathtt{i}}}\tto 
P_{\mathbbm{1}_{\mathtt{i}}}/K$ we obtain the exact sequence
\begin{displaymath}
\mathrm{F}\,K\hookrightarrow \mathrm{F}\, P_{\mathbbm{1}_{\mathtt{i}}}\tto
\mathrm{F}\,P_{\mathbbm{1}_{\mathtt{i}}}/K.
\end{displaymath}
As $\mathrm{F}\,P_{\mathbbm{1}_{\mathtt{i}}}/K=0$, 
we see that $\mathrm{F}\,K\cong \mathrm{F}\, P_{\mathbbm{1}_{\mathtt{i}}}$
is an indecomposable projective and hence has simple top. 
Applying $\mathrm{F}$ to the exact sequence
$\mathrm{rad}\,K\hookrightarrow K\tto
\mathrm{top}\,K$ we obtain 
the exact sequence
\begin{displaymath}
\mathrm{F}\, \mathrm{rad}\,K\hookrightarrow \mathrm{F}\, K\tto
\mathrm{F}\,\mathrm{top}\,K.
\end{displaymath}
As $\mathrm{F}\,K$ has simple top by the above, we obtain that
$\mathrm{F}\,\mathrm{top}\,K$
has simple top. By construction, $\mathrm{top}\,K$ is semi-simple 
and none of its submodules are annihilated by $\mathrm{F}$. 
Therefore $\mathrm{top}\,K$ is simple, which implies \eqref{lem7.01}
and also \eqref{lem7.02}.

For $\mathrm{F}\in\mathcal{R}$, the projective module 
$P_{\mathrm{F}}$ surjects onto  the nontrivial module 
$\mathrm{F}\, L_{\mathrm{G}_{\mathcal{R}}}$ by the above.
Hence $L_{\mathrm{F}}$ occurs in the top of 
$\mathrm{F}\, L_{\mathrm{G}_{\mathcal{R}}}$ and thus
\eqref{lem7.03} follows from Corollary~\ref{cor121}.
For $\mathrm{F}\in\mathcal{R}$ we have 
$\mathrm{F}\, L_{\mathrm{G}_{\mathcal{R}}}\neq 0$ and hence
\eqref{lem7.04} follows from Lemma~\ref{lem11}.

From \eqref{lem7.04} we have $\mathrm{G}_{\mathcal{R}}\leq_R
\mathrm{G}_{\mathcal{R}}^*$. Assume that 
$\mathrm{G}_{\mathcal{R}}^*\not\in\mathcal{R}$ and let
$\tilde{\mathcal{R}}$ be the right cell containing
$\mathrm{G}_{\mathcal{R}}^*$. By Lemma~\ref{lem122} we have
$\mathrm{G}_{\mathcal{R}}^*L_{\mathrm{G}_{\mathcal{R}}}\neq 0$,
which implies that $K_{\mathcal{R}}\subset K_{\tilde{\mathcal{R}}}$.
If $K_{\mathcal{R}}=K_{\tilde{\mathcal{R}}}$, then 
$L_{\mathrm{G}_{\mathcal{R}}}=L_{\mathrm{G}_{\tilde{\mathcal{R}}}}$
and hence $\mathcal{R}=\tilde{\mathcal{R}}$, which implies \eqref{lem7.05}.
If $K_{\mathcal{R}}\subsetneq K_{\tilde{\mathcal{R}}}$, then
from \eqref{lem7.1} we have 
$\mathrm{G}_{\mathcal{R}}\, L_{\mathrm{G}_{\tilde{\mathcal{R}}}}=0$. 
As $\mathrm{Ann}_{\mathcal{C}}
(L_{\mathrm{G}_{\tilde{\mathcal{R}}}})$ is a coideal with respect to
$\leq_R$, it follows that $L_{\mathrm{G}_{\tilde{\mathcal{R}}}}$
is annihilated by $\mathrm{G}_{\mathcal{R}}^*$. This contradicts
\eqref{lem7.2} and hence \eqref{lem7.05} follows. The proof is complete.
\end{proof}
 
For simplicity we set $L=L_{\mathrm{G}_{\mathcal{R}}}$.
For $\mathtt{j}\in\cC$ denote by 
$\mathcal{D}_{\mathcal{R},\mathtt{j}}$ the full subcategory of
$\mathbf{P}_{\mathtt{i}}(\mathtt{j})$ with objects
$\mathrm{G}\, L$, $\mathrm{G}\in \mathcal{R}\cap 
\mathcal{C}_{\mathtt{i},\mathtt{j}}$. As each $\mathrm{G}\, L$
is a quotient of $P_{\mathrm{G}}$ and $2$-morphisms in $\cC$
surject onto homomorphisms between projective modules in
$\mathbf{P}_{\mathtt{i}}(\mathtt{j})$ (see Subsection~\ref{s2.3}), it 
follows that $2$-morphisms in $\cC$ surject onto homomorphisms between 
the various $\mathrm{G}\, L$. 

\begin{lemma}\label{lem14}
For every $\mathrm{F}\in\mathcal{C}$ and $\mathrm{G}\in \mathcal{R}$,
the module $\mathrm{F}\circ \mathrm{G}\,L$ is isomorphic to a direct
sum of modules of the form $\mathrm{H}\,L$, $\mathrm{H}\in \mathcal{R}$.
\end{lemma}

\begin{proof}
Any $\mathrm{H}$ occurring as a direct summand of 
$\mathrm{F}\circ \mathrm{G}$ satisfies $\mathrm{H}\geq_R \mathrm{G}$.
On the other hand, $\mathrm{H}\,L\neq 0$ implies
$\mathrm{H}^*\leq_L \mathrm{G}_{\mathcal{R}}$ by Lemma~\ref{lem11}.
This is equivalent to $\mathrm{H}\leq_R \mathrm{G}_{\mathcal{R}}^*$.
By Proposition~\ref{lem7}\eqref{lem7.05}, we have 
$\mathrm{G}_{\mathcal{R}}^*\in \mathcal{R}$. Thus
$\mathrm{H}\in \mathcal{R}$, as claimed.
\end{proof}

\begin{lemma}\label{lem15}
For every $\mathrm{F},\mathrm{H}\in \mathcal{R}\cap
\mathcal{C}_{\mathtt{i},\mathtt{j}}$ we have
\begin{displaymath}
\dim\mathrm{Hom}_{\overline{\ccC}(\mathtt{i},\mathtt{j})}
(\mathrm{F}\,L,\mathrm{H}\,L)=[\mathrm{H}\,L:L_{\mathrm{F}}].
\end{displaymath}
\end{lemma}

\begin{proof}
Let $k$ denote the multiplicity of $\mathrm{G}_{\mathcal{R}}$
as a direct summand of $\mathrm{H}^*\circ\mathrm{F}$. Then for the right
hand side we have
\begin{displaymath}
\begin{array}{rcl}
[\mathrm{H}\,L:L_{\mathrm{F}}]&=&\dim 
\mathrm{Hom}_{\overline{\ccC}(\mathtt{i},\mathtt{j})}
(\mathrm{F}\,P_{\mathbbm{1}_{\mathtt{i}}},\mathrm{H}\,L)\\
\text{(by adjunction)}&=&\dim 
\mathrm{Hom}_{\overline{\ccC}(\mathtt{i},\mathtt{i})}
(\mathrm{H}^*\circ\mathrm{F}\,P_{\mathbbm{1}_{\mathtt{i}}},L)\\
&=&k.
\end{array}
\end{displaymath}
At the same time, by adjunction, for the left hand side we have 
\begin{equation}\label{eq01}
\dim\mathrm{Hom}_{\overline{\ccC}(\mathtt{i},\mathtt{j})}
(\mathrm{F}\,L,\mathrm{H}\,L)=
\dim\mathrm{Hom}_{\overline{\ccC}(\mathtt{i},\mathtt{i})}
(\mathrm{H}^*\circ\mathrm{F}\,L,L).
\end{equation}
From Proposition~\ref{lem7}\eqref{lem7.02} it follows that the 
right hand side of \eqref{eq01} is at least $k$. On the other hand,
\begin{displaymath}
\dim\mathrm{Hom}_{\overline{\ccC}(\mathtt{i},\mathtt{j})}
(\mathrm{F}\,L,\mathrm{H}\,L)\leq 
\mathrm{Hom}_{\overline{\ccC}(\mathtt{i},\mathtt{j})}
(\mathrm{F}\,P_{\mathbbm{1}_{\mathtt{i}}},\mathrm{H}\,L)
=[\mathrm{H}\,L:L_{\mathrm{F}}]=k,
\end{displaymath}
which completes the proof.
\end{proof}

For $\mathrm{F}\in\mathcal{R}$ consider the short exact sequence
\begin{equation}\label{eq777}
\mathrm{Ker}_{\mathrm{F}}\hookrightarrow
P_{\mathrm{F}}\tto \mathrm{F}\, L,
\end{equation}
given by Proposition~\ref{lem7}\eqref{lem7.02}. Set
\begin{displaymath}
\mathrm{Ker}_{\mathcal{R},\mathtt{j}}=
\bigoplus_{\mathrm{F}\in \mathcal{R}\cap\mathcal{C}_{\mathtt{i},\mathtt{j}}}
\mathrm{Ker}_{\mathrm{F}},\quad\quad
P_{\mathcal{R},\mathtt{j}}=
\bigoplus_{\mathrm{F}\in \mathcal{R}\cap\mathcal{C}_{\mathtt{i},\mathtt{j}}}
P_{\mathrm{F}},\quad\quad
Q_{\mathcal{R},\mathtt{j}}=
\bigoplus_{\mathrm{F}\in \mathcal{R}\cap\mathcal{C}_{\mathtt{i},\mathtt{j}}}
\mathrm{F}\, L.
\end{displaymath}

\begin{lemma}\label{lem16}
The module $\mathrm{Ker}_{\mathcal{R},\mathtt{j}}$ is stable under any
endomorphism of $P_{\mathcal{R},\mathtt{j}}$.
\end{lemma}

\begin{proof}
Let $\mathrm{F},\mathrm{H}\in\mathcal{R}
\cap\mathcal{C}_{\mathtt{i},\mathtt{j}}$ and
$\varphi:P_{\mathrm{F}}\to P_{\mathrm{H}}$ be a homomorphism. 
It is enough to show that $\varphi(\mathrm{Ker}_{\mathrm{F}})
\subset \mathrm{Ker}_{\mathrm{H}}$. Assume this is false. 
Composing $\varphi$ with the natural projection onto
$Q_{\mathrm{H}}$ we obtain a homomorphism from 
$P_{\mathrm{F}}$ to $Q_{\mathrm{H}}$ which does not factor
through $Q_{\mathrm{F}}$. However, the existence of such homomorphism
contradicts Lemma~\ref{lem15}. This implies the claim.
\end{proof}

Now we are ready to define the cell $2$-representation  
$\mathbf{C}_{\mathcal{R}}$ of $\cC$ corresponding to $\mathcal{R}$. Define 
$\mathbf{C}_{\mathcal{R}}(\mathtt{j})$ to be the full subcategory of
$\mathbf{P}_{\mathtt{i}}(\mathtt{j})$ which consists of all modules
$M$ admitting a two step resolution $X_1\to X_0\tto M$,
where $X_1,X_0\in\mathrm{add}(Q_{\mathcal{R},\mathtt{j}})$.

\begin{lemma}\label{lem17}
The category $\mathbf{C}_{\mathcal{R}}(\mathtt{j})$ is equivalent
to $\mathcal{D}_{\mathcal{R},\mathtt{j}}^{\mathrm{op}}\text{-}\mathrm{mod}$.
\end{lemma}

\begin{proof}
Consider first the full subcategory $\mathcal{X}$ of
$\mathbf{P}_{\mathtt{i}}(\mathtt{j})$ which consists of all modules
$M$ admitting a two step resolution $X_1\to X_0\tto M$,
where $X_1,X_0\in\mathrm{add}(P_{\mathcal{R}})$.
By \cite[Section~5]{Au}, the category $\mathcal{X}$ is equivalent to
$\mathrm{End}_{\mathcal{C}_{\mathtt{i},\mathtt{j}}^{\mathrm{op}}}
(P_{\mathcal{R},\mathtt{j}})^{\mathrm{op}}\text{-}\mathrm{mod}$.
By Lemma~\ref{lem16}, the algebra 
$\mathrm{End}_{\mathcal{C}_{\mathtt{i},\mathtt{j}}^{\mathrm{op}}}
(Q_{\mathcal{R},\mathtt{j}})$ is the quotient of 
$\mathrm{End}_{\mathcal{C}_{\mathtt{i},\mathtt{j}}^{\mathrm{op}}}
(P_{\mathcal{R},\mathtt{j}})$ by a two-sided ideal. It is easy to see that 
the standard embedding of 
$\mathrm{End}_{\mathcal{C}_{\mathtt{i},\mathtt{j}}^{\mathrm{op}}}
(Q_{\mathcal{R},\mathtt{j}})^{\mathrm{op}}\text{-}\mathrm{mod}\cong
\mathcal{D}_{\mathcal{R},\mathtt{j}}^{\mathrm{op}}\text{-}\mathrm{mod}$
into $\mathcal{X}$ coincides with
$\mathbf{C}_{\mathcal{R}}(\mathtt{j})$. The claim follows.
\end{proof}

\begin{theorem}[Construction of right cell $2$-representations]
\label{thm19}
Restriction from  $\mathbf{P}_{\mathtt{i}}$ defines the structure of a 
$2$-representation of $\cC$ on  $\mathbf{C}_{\mathcal{R}}$.
\end{theorem}

\begin{proof}
From Lemma~\ref{lem14} it follows that for any $\mathrm{F}\in
\mathcal{C}_{\mathtt{j},\mathtt{k}}$ we have
$\mathrm{F}\, \mathbf{C}_{\mathcal{R}}(\mathtt{j})\subset
\mathbf{C}_{\mathcal{R}}(\mathtt{k})$. The claim follows.
\end{proof}

The $2$-representation $\mathbf{C}_{\mathcal{R}}$ 
constructed in Theorem~\ref{thm19} is called the 
{\em right cell $2$-representation} corresponding to $\mathcal{R}$.
Note that the inclusion of $\mathbf{C}_{\mathcal{R}}$ into
$\mathbf{P}_{\mathtt{i}}$ is only right exact in general.

\begin{example}\label{exm18}
{\rm Consider the category $\cS_2$ from Example~\ref{exm1}. For the
cell representation $\mathbf{C}_{\{\mathbbm{1}_{\mathtt{i}}\}}$
we have $G_{\{\mathbbm{1}_{\mathtt{i}}\}}=\mathbbm{1}_{\mathtt{i}}$,
which implies that $\mathbf{C}_{\{\mathbbm{1}_{\mathtt{i}}\}}(\mathtt{i})=
\mathbb{C}\text{-}\mathrm{mod}$; 
$\mathbf{C}_{\{\mathbbm{1}_{\mathtt{i}}\}}(\mathrm{F})=0$
and $\mathbf{C}_{\{\mathbbm{1}_{\mathtt{i}}\}}(f)=0$ for 
$f=\alpha,\beta,\gamma$. For the
cell representation $\mathbf{C}_{\{\mathrm{F}\}}$
we have $G_{\{\mathrm{F}\}}=\mathrm{F}$,
which implies that $\mathbf{C}_{\{\mathrm{F}\}}(\mathtt{i})=
D\text{-}\mathrm{mod}$, 
$\mathbf{C}_{\{\mathrm{F}\}}(\mathrm{F})=\mathrm{F}$
and $\mathbf{C}_{\{\mathrm{F}\}}(f)=f$ for 
$f=\alpha,\beta,\gamma$.
}
\end{example}

\subsection{Homomorphisms from a cell $2$-representation}\label{s3.5}

Consider a right cell $\mathcal{R}$ and let $\mathtt{i}\in\cC$ be such 
that $\mathrm{G}_{\mathcal{R}}\in\mathcal{C}_{\mathtt{i},\mathtt{i}}$.
Let further $\mathrm{F}\in\cC(\mathtt{i},\mathtt{i})$
and $\alpha:\mathrm{F}\to \mathrm{G}_{\mathcal{R}}$ be such that 
$\mathbf{P}_{\mathtt{i}}(\alpha):\mathrm{F}\,P_{\mathbbm{1}_{\mathtt{i}}}
\to\mathrm{G}_{\mathcal{R}}\,P_{\mathbbm{1}_{\mathtt{i}}}$
gives a projective presentation of $L_{\mathrm{G}_{\mathcal{R}}}$.

\begin{theorem}\label{prop21}
Let $\mathbf{M}$ be a $2$-representation of $\cC$. Denote by
$\Theta=\Theta_{\mathcal{R}}^{\mathbf{M}}$ the cokernel of 
$\mathbf{M}(\alpha)$. 
\begin{enumerate}[$($a$)$]
\item\label{prop21.1} The functor $\Theta$ is a right exact 
endofunctor of $\mathbf{M}(\mathtt{i})$. 
\item\label{prop21.2} For every morphism $\Psi$ 
from $\mathbf{C}_{\mathcal{R}}$ to $\mathbf{M}$ we have
$\Psi(L_{\mathrm{G}_{\mathcal{R}}})\in
\Theta(\mathbf{M}(\mathtt{i}))$.
\item\label{prop21.3} For every $M\in\Theta(\mathbf{M}(\mathtt{i}))$
there is a unique morphism $\Psi^M$ from $\mathbf{C}_{\mathcal{R}}$ 
to $\mathbf{M}$ given by a collection of right exact functors such that
$\Psi^M$ sends $L_{\mathrm{G}_{\mathcal{R}}}$ to $M$.
\item\label{prop21.4} The correspondence $M\mapsto \Psi^M$ is 
functorial in $M$ in the image $\Theta(\mathbf{M}(\mathtt{i}))$
of $\Theta$.
\end{enumerate}
\end{theorem}

\begin{proof}
Both $\mathbf{M}(\mathrm{F})$ and $\mathbf{M}(\mathrm{G}_{\mathcal{R}})$
are exact functors as $\cC$ is a fiat category and $\mathbf{M}$ is
a $2$-functor. The functor $\Theta$ is the cokernel of a homomorphism
between two exact functors and hence is right exact by the Snake lemma.
This proves claim \eqref{prop21.1}. Claim \eqref{prop21.2} follows
from the definitions.

To prove claims \eqref{prop21.3} and \eqref{prop21.4} choose
$M\in\Theta(\mathbf{M}(\mathtt{i}))$ such that $M=\Theta\, N$ for some 
$N\in \mathbf{M}(\mathtt{i})$. Consider the morphism $\Phi^N$ given by
Proposition~\ref{prop3}. As $\Phi^N$ is a morphism of $2$-representations, 
$\Phi^N(L_{\mathrm{G}_{\mathcal{R}}})=\Theta\, N=M$. The restriction 
$\Psi^M$ of $\Phi^N$ to $\mathbf{C}_{\mathcal{R}}$
is a morphism from $\mathbf{C}_{\mathcal{R}}$ to $\mathbf{M}$.
Now the existence parts of \eqref{prop21.3} and \eqref{prop21.4} follow 
from Proposition~\ref{prop3}. To prove uniqueness, we note that,
for every $\mathtt{j}\in\cC$, every projective in 
$\mathbf{C}_{\mathcal{R}}(\mathtt{j})$ has the form
$\mathrm{F}\, L_{\mathrm{G}_{\mathcal{R}}}$ for some
$\mathrm{F}\in\cC(\mathtt{i},\mathtt{j})$ and every morphism between
projectives comes from a $2$-morphism of $\cC$ (see Subsection~\ref{s3.4}).
As any morphism from $\mathbf{C}_{\mathcal{R}}$ to $\mathbf{M}$ is
a natural transformation of $2$-functors, the value of this transformation
on $L_{\mathrm{G}_{\mathcal{R}}}$ uniquely determines its value on
all other modules. This implies the uniqueness claim and completes the proof.
\end{proof}

\subsection{A canonical quotient of $P_{\mathbbm{1}_{\mathtt{i}}}$
associated with $\mathcal{R}$}\label{s3.6}

Fix $\mathtt{i}\in\cC$. Let $\mathcal{R}$ be a right cell in 
$\mathcal{C}$ such that $\mathcal{R}\cap 
\mathcal{C}_{\mathtt{i},\mathtt{j}}\neq \varnothing$ for
some $\mathtt{j}\in\cC$. Denote by $\Delta_{\mathcal{R}}$
the unique minimal quotient of $P_{\mathbbm{1}_{\mathtt{i}}}$
such that the composition $K_{\mathcal{R}}\hookrightarrow
P_{\mathbbm{1}_{\mathtt{i}}}\tto \Delta_{\mathcal{R}}$ is nonzero.

\begin{proposition}\label{prop28}
For every $\mathrm{F}\in \mathcal{R}$ the image of the a unique 
(up to scalar) nonzero homomorphism 
$\varphi:P_{\mathbbm{1}_{\mathtt{i}}}\to \mathrm{F}^*\, L_{\mathrm{F}}$ 
is isomorphic to $\Delta_{\mathcal{R}}$.
\end{proposition}

\begin{proof}
The existence of $\varphi$ is given by Lemma~\ref{lem122}.
Let $Y$ denote the image of $\varphi$.
Assume that $\mathrm{F}\in\mathcal{C}_{\mathtt{i},\mathtt{j}}$.
For $X\in\{P_{\mathbbm{1}_{\mathtt{i}}},\Delta_{\mathcal{R}}\}$
we have, by adjunction,
\begin{displaymath}
\mathrm{Hom}_{\overline{\ccC}(\mathtt{i},\mathtt{i})}
(X,\mathrm{F}^*\, L_{\mathrm{F}})=
\mathrm{Hom}_{\overline{\ccC}(\mathtt{i},\mathtt{j})}
(\mathrm{F}\, X,L_{\mathrm{F}})=\Bbbk
\end{displaymath}
as $\mathrm{F}\, X$ is a nontrivial quotient of $P_{\mathrm{F}}$
(see Proposition~\ref{lem7}). By construction, $Y$ has simple top
isomorphic to $L_{\mathbbm{1}_{\mathtt{i}}}$ and, by the above, 
the latter module occurs in $\mathrm{F}^*\, L_{\mathrm{F}}$ with 
multiplicity one.
Since $\Delta_{\mathcal{R}}$ also has simple top isomorphic 
to $L_{\mathbbm{1}_{\mathtt{i}}}$, it follows that
the image of any nonzero map from $\Delta_{\mathcal{R}}$
to $Y$ covers the top of $Y$ and hence is surjective. To complete
the proof it is left to show that the image of any nonzero map 
from $\Delta_{\mathcal{R}}$ to $Y$ is injective.

By construction, $L_{\mathrm{G}_{\mathcal{R}}}$ is the simple
socle of $\Delta_{\mathcal{R}}$. Let $N$ denote the cokernel of
$L_{\mathrm{G}_{\mathcal{R}}}\hookrightarrow \Delta_{\mathcal{R}}$.
Similarly to the previous paragraph, we have
\begin{displaymath}
\mathrm{Hom}_{\overline{\ccC}(\mathtt{i},\mathtt{i})}
(N,\mathrm{F}^*\, L_{\mathrm{F}})=
\mathrm{Hom}_{\overline{\ccC}(\mathtt{i},\mathtt{j})}
(\mathrm{F}\, N,L_{\mathrm{F}})=0
\end{displaymath}
since all composition factors of $N$ are annihilated by $\mathrm{F}$
(by Proposition~\ref{lem7}\eqref{lem7.1}).
The claim follows.
\end{proof}

We complete this section with the following collection of useful 
facts:

\begin{lemma}\label{lem29}
\begin{enumerate}[$($a$)$]
\item\label{lem29.1} For any 
$\mathrm{F},\mathrm{G}\in \mathcal{C}_{\mathtt{i},\mathtt{j}}$
we have $[\mathrm{F}^*\,L_{\mathrm{G}}:L_{\mathbbm{1}_{\mathtt{i}}}]\neq 0$ 
if and only if $\mathrm{F}=\mathrm{G}$.
\item\label{lem29.2} For any $\mathrm{F}\in\mathcal{C}$ we have
$\mathrm{F}\sim_{LR}\mathrm{F}^*$.
\end{enumerate} 
\end{lemma}

\begin{proof}
Using adjunction, we have 
\begin{displaymath}
\mathrm{Hom}_{\overline{\ccC}(\mathtt{i},\mathtt{i})}
(P_{\mathbbm{1}_{\mathtt{i}}},\mathrm{F}^*\, L_{\mathrm{G}})=
\mathrm{Hom}_{\overline{\ccC}(\mathtt{i},\mathtt{j})}
(P_{\mathrm{F}},L_{\mathrm{G}})=
\begin{cases}
\Bbbk, &  \mathrm{F}=\mathrm{G};\\
 0, & \text{otherwise},
\end{cases}
\end{displaymath}
which proves \eqref{lem29.1}.
  
To prove  \eqref{lem29.2} let $\mathcal{R}$ be the right cell 
containing $\mathrm{F}$. Then we have $\mathrm{F}\sim_R
\mathrm{G}_{\mathcal{R}}$ and hence $\mathrm{F}^*\sim_L
\mathrm{G}_{\mathcal{R}}^*$. At the same time
$\mathrm{G}_{\mathcal{R}}\sim_R \mathrm{G}_{\mathcal{R}}^*$
by Proposition~\ref{lem7}\eqref{lem7.05}. Claim \eqref{lem29.2}
follows and the proof is complete.
\end{proof}

\subsection{Regular cells}\label{s3.7}

We denote by $\star$ the usual product of binary relations.

\begin{lemma}\label{lem31}
We have $\leq_{LR}=\leq_R\star\leq_L=\leq_L\star\leq_R$.
\end{lemma}

\begin{proof}
Obviously the product of  $\leq_R$ and $\leq_L$ (in any order)
is contained in $\leq_{LR}$. On the other hand,
for $\mathrm{F},\mathrm{G}\in \mathcal{C}$ we have 
$\mathrm{F}\leq_{LR}\mathrm{G}$ if and only if there exist
$\mathrm{H},\mathrm{K}\in \mathcal{C}$ such that $\mathrm{G}$
occurs as a direct summand of $\mathrm{H}\circ\mathrm{F}\circ\mathrm{K}$.
This means that there is a direct summand $\mathrm{L}$ of
$\mathrm{H}\circ\mathrm{F}$ such that $\mathrm{G}$
occurs as a direct summand of $\mathrm{L}\circ\mathrm{K}$.
By definition, we have $\mathrm{F}\leq_R\mathrm{L}$ and
$\mathrm{L}\leq_R\mathrm{G}$. This implies that $\leq_{LR}$
is contained in $\leq_R\star\leq_L$ and hence
$\leq_{LR}$ coincides with $\leq_R\star\leq_L$. Similarly 
$\leq_{LR}$ coincides with $\leq_L\star\leq_R$ and the claim 
of the lemma follows.
\end{proof}

A two-sided cell $\mathcal{Q}$ is called {\em regular} provided
that any two different right cells inside $\mathcal{Q}$ are not
comparable with respect to the right order. From 
Lemma~\ref{lem29}\eqref{lem29.2} it follows that $\mathcal{Q}$
is regular if and only if any two different left cells inside 
$\mathcal{Q}$ are not comparable with respect to the left order.
A right (left) cell is called {\em regular} if it belongs to
a regular two-sided cell. An element $\mathrm{F}$ is called
{\em regular} if it belongs to a regular two-sided cell.

\begin{proposition}[Structure of regular two-sided cells]\label{prop32}
Let $\mathcal{Q}$ be a regular two-sided cell.
\begin{enumerate}[$($a$)$]
\item\label{prop32.1} 
For any right cell $\mathcal{R}$ in $\mathcal{Q}$ and left cell
$\mathcal{L}$ in $\mathcal{Q}$ we have
$\mathcal{L}\cap \mathcal{R}\neq \varnothing$.
\item\label{prop32.2}
Let $\sim^{\mathcal{Q}}_{R}$ and $\sim^{\mathcal{Q}}_{L}$ denote the
restrictions of $\sim_R$ and $\sim_L$ to $\mathcal{Q}$,
respectively. Then $\mathcal{Q}\times \mathcal{Q}=
\sim^{\mathcal{Q}}_{R}\star\sim^{\mathcal{Q}}_{L}=
\sim^{\mathcal{Q}}_{L}\star\sim^{\mathcal{Q}}_{R}$.
\end{enumerate}
\end{proposition}

\begin{proof}
If $\mathrm{F}\in \mathcal{R}$ and $\mathrm{G}\in \mathcal{L}$, then
there exist $\mathrm{H},\mathrm{K}\in \mathcal{C}$ such that
$\mathrm{G}$ occurs as a direct summand of
$\mathrm{H}\circ\mathrm{F}\circ\mathrm{K}$. This means that there
exists a direct summand $\mathrm{N}$ of $\mathrm{H}\circ\mathrm{F}$
such that $\mathrm{G}$ occurs as a direct summand of
$\mathrm{N}\circ\mathrm{K}$. Then $\mathrm{N}\geq_R\mathrm{F}$
and $\mathrm{N}\leq_L \mathrm{G}$. As $\mathrm{F}\sim_{LR}\mathrm{G}$
it follows that $\mathrm{N}\in\mathcal{Q}$. Since $\mathcal{Q}$ is
regular, it follows that $\mathrm{N}\in \mathcal{R}$ and 
$\mathrm{N}\in \mathcal{L}$ proving \eqref{prop32.1}.

To prove \eqref{prop32.2} consider $\mathrm{F},\mathrm{G}\in \mathcal{Q}$.
By \eqref{prop32.1}, there exist $\mathrm{H}\in \mathcal{Q}$
such that $\mathrm{H}\sim_{L}\mathrm{F}$ and $\mathrm{H}\sim_{R}\mathrm{G}$.
Similarly, there exist $\mathrm{K}\in \mathcal{Q}$
such that $\mathrm{K}\sim_{R}\mathrm{F}$ and $\mathrm{K}\sim_{L}\mathrm{G}$.
Then we have $(\mathrm{F},\mathrm{G})=(\mathrm{F},\mathrm{H})\star
(\mathrm{H},\mathrm{G})$ and $(\mathrm{F},\mathrm{G})=
(\mathrm{F},\mathrm{K})\star (\mathrm{K},\mathrm{G})$ 
proving \eqref{prop32.2}. 
\end{proof}

For a regular right cell $\mathcal{R}$ the corresponding module
$\Delta_{\mathcal{R}}$ has the following property.

\begin{proposition}\label{prop23}
Let $\mathcal{R}$ be a regular right cell and $M$ the cokernel of
$L_{\mathrm{G}_{\mathcal{R}}}\hookrightarrow \Delta_{\mathcal{R}}$.
Then for any composition factor $L_{\mathrm{F}}$ of $M$
we have $\mathrm{F}<_R\mathrm{G}_{\mathcal{R}}$
and $\mathrm{F}<_L\mathrm{G}_{\mathcal{R}}$.
\end{proposition}

\begin{proof}[Proof of Proposition~\ref{prop23}.]
Let $\mathrm{F}\in\mathcal{C}$ be such that $L_{\mathrm{F}}$ is a 
composition factor of $M$. As $\Delta_{\mathcal{R}}$ is
a submodule of $\mathrm{G}^*_{\mathcal{R}}\,L_{\mathrm{G}_{\mathcal{R}}}$
(by Proposition~\ref{prop28}),
from Lemma~\ref{lem12}\eqref{lem12.1} it follows that
$\mathrm{F}\leq_R \mathrm{G}_{\mathcal{R}}$.
 
Consider $\mathcal{I}:=\{\mathrm{H}\in\mathcal{C}:
\mathrm{H}\leq_R \mathrm{G}_{\mathcal{R}}\}$. Then
$\mathcal{I}$ is an ideal with respect to $\leq_R$.
Assume that $\mathrm{G}_{\mathcal{R}}\in
\mathcal{C}_{\mathtt{i},\mathtt{i}}$ and consider
the $2$-subrepresentation $\mathbf{P}_{\mathtt{i}}^{\mathcal{I}}$ of 
$\mathbf{P}_{\mathtt{i}}$. Then $\mathrm{F}\in
\mathcal{C}_{\mathtt{i},\mathtt{i}}$ and
$\Delta_{\mathcal{R}}\in \mathbf{P}_{\mathtt{i}}^{\mathcal{I}}(\mathtt{i})$.
Using adjunction, we have: 
\begin{displaymath}
0\neq\mathrm{Hom}_{\overline{\ccC}(\mathtt{i},\mathtt{i})}
(\mathrm{F}\,P_{\mathbbm{1}_{\mathtt{i}}}^{\mathcal{I}},
\Delta_{\mathcal{R}})=
\mathrm{Hom}_{\overline{\ccC}(\mathtt{i},\mathtt{i})}
(P_{\mathbbm{1}_{\mathtt{i}}}^{\mathcal{I}},
\mathrm{F}^*\,\Delta_{\mathcal{R}}).
\end{displaymath}
This yields $\mathrm{F}^*\,\Delta_{\mathcal{R}}\neq 0$. The module
$\mathrm{F}^*\,\Delta_{\mathcal{R}}$ on the one hand belongs to
$\mathbf{P}_{\mathtt{i}}^{\mathcal{I}}(\mathtt{i})$
(by Lemma~\ref{lem12}\eqref{lem12.1}), on the other hand 
is a quotient of $\mathrm{F}^*\,P_{\mathbbm{1}_{\mathtt{i}}}$
(as $\Delta_{\mathcal{R}}$ is a quotient of $P_{\mathbbm{1}_{\mathtt{i}}}$
and $\mathrm{F}^*$ is exact). The module 
$\mathrm{F}^*\,P_{\mathbbm{1}_{\mathtt{i}}}$ 
has simple top $L_{\mathrm{F}^*}$. This implies
$\mathrm{F}^*\leq_R \mathrm{G}_{\mathcal{R}}$ 
by Lemma~\ref{lem12}\eqref{lem12.1} and thus
$\mathrm{F}\leq_L \mathrm{G}_{\mathcal{R}}^*\in\mathcal{R}$
(see Proposition~\ref{lem7}\eqref{lem7.05}). 

This leaves us with two possibilities: either
$\mathrm{F}\not\sim_{LR}\mathrm{G}_{\mathcal{R}}$, in which case
we have both  $\mathrm{F}<_{R}\mathrm{G}_{\mathcal{R}}$
and $\mathrm{F}<_{L}\mathrm{G}_{\mathcal{R}}$, as desired; or
$\mathrm{F}\sim_{LR}\mathrm{G}_{\mathcal{R}}$, in which case
we have both  $\mathrm{F}\sim_{L}\mathrm{G}_{\mathcal{R}}$
and $\mathrm{F}\sim_{R}\mathrm{G}_{\mathcal{R}}$ since
$\mathcal{R}$ is regular. In the latter case we, however,
have $\mathrm{G}_{\mathcal{R}}^*\,L_{\mathrm{F}}\neq 0$
by Lemma~\ref{lem11}, which contradicts 
Proposition~\ref{lem7}\eqref{lem7.1}. This completes the proof.
\end{proof}

A two-sided cell $\mathcal{Q}$ is called {\em strongly regular} if it 
is regular and for every left cell $\mathcal{L}$ and right cell
$\mathcal{R}$ in $\mathcal{Q}$ we have $|\mathcal{L}\cap\mathcal{R}|=1$.
A left (right) cell is {\em strongly regular} if it is contained in a
strongly regular two-sided cell.

\begin{proposition}[Structure of strongly regular right
cells]\label{prop24}
Let $\mathcal{R}$ be a strongly regular right cell. Then we have:
\begin{enumerate}[$($a$)$]
\item\label{prop24.1} $\mathrm{G}_{\mathcal{R}}\cong
\mathrm{G}_{\mathcal{R}}^*$.
\item\label{prop24.2}  If $\mathrm{F}\in \mathcal{R}$
satisfies $\mathrm{F}\cong\mathrm{F}^*$, then 
$\mathrm{F}=\mathrm{G}_{\mathcal{R}}$.
\item\label{prop24.3} If $\mathrm{F}\in \mathcal{R}$
and $\mathrm{G}\sim_{{L}}\mathrm{F}$ is such that 
$\mathrm{G}\cong\mathrm{G}^*$, then $\mathrm{G}\,L_{\mathrm{F}}\neq0$
and every simple occurring both in the top and in the socle of
$\mathrm{G}\,L_{\mathrm{F}}$ is isomorphic to $L_{\mathrm{F}}$.
\end{enumerate}
\end{proposition}

\begin{proof}
Claim \eqref{prop24.1} follows from the strong regularity of
$\mathcal{R}$ and Proposition~\ref{lem7}\eqref{lem7.05}.
Claim \eqref{prop24.2} follows directly from the strong
regularity of $\mathcal{R}$.

Let us prove claim \eqref{prop24.3}. That 
$\mathrm{G}\,L_{\mathrm{F}}\neq0$ follows from 
Lemma~\ref{lem11}. If some $L_{\mathrm{H}}$ occurs in the
top of $\mathrm{G}\,L_{\mathrm{F}}\neq0$ then, using adjunction
and $\mathrm{G}\cong\mathrm{G}^*$, we get $\mathrm{G}\,L_{\mathrm{H}}\neq 0$.
The latter implies $\mathrm{G}\sim_L\mathrm{H}$ by
Lemma~\ref{lem11}. At the same time $\mathrm{H}\sim_R\mathrm{F}$
by Corollary~\ref{cor121}. Hence $\mathrm{H}=\mathrm{F}$ because of the 
strong regularity of $\mathcal{R}$. This completes the proof.
\end{proof}

\section{The $2$-category of a two-sided cell}\label{s5}

\subsection{The quotient associated with a two-sided cell}\label{s5.1}

Let $\mathcal{Q}$ be a two-sided cell in $\mathcal{C}$.
Denote by $\cI_{\mathcal{Q}}$ the $2$-ideal of $\cC$ generated by 
$\mathrm{F}$ and $\mathrm{id}_{\mathrm{F}}$ for all 
$\mathrm{F}\not\leq_{LR}\mathcal{Q}$. In other words, for every
$\mathtt{i},\mathtt{j}\in\cC$ we have that 
$\cI_{\mathcal{Q}}(\mathtt{i},\mathtt{j})$
is the ideal of $\cC(\mathtt{i},\mathtt{j})$ consisting of all 
$2$-morphisms which factor through a direct sum of $1$-morphisms
of the form $\mathrm{F}$, where $\mathrm{F}\not\leq_{LR}\mathcal{Q}$.
Taking the quotient we obtain the $2$-category $\cC/\cI_{\mathcal{Q}}$.

\begin{lemma}\label{lem61}
Let $\mathcal{R}\subset \mathcal{Q}$ be a right cell. Then
$\cI_{\mathcal{Q}}$ annihilates the cell $2$-rep\-re\-sen\-ta\-ti\-on
$\mathbf{C}_{\mathcal{R}}$. In particular, $\mathbf{C}_{\mathcal{R}}$
carries the natural structure of a $2$-rep\-re\-sen\-ta\-ti\-on
of  $\cC/\cI_{\mathcal{Q}}$.
\end{lemma}

\begin{proof}
This follows from the construction and Lemma~\ref{lem11}.
\end{proof}

The construction of $\cC/\cI_{\mathcal{Q}}$ is analogous to
constructions from \cite{Be,Os}.

\subsection{The $2$-category associated with $\mathcal{Q}$}\label{s5.2}

Denote by $\cC_{\mathcal{Q}}$ the full $2$-subcategory 
of $\cC/\cI_{\mathcal{Q}}$, closed under isomorphisms, generated by 
the identity morphisms $\mathbbm{1}_{\mathtt{i}}$, $\mathtt{i}\in\cC$, 
and $\mathrm{F}\in\mathcal{Q}$. We will call $\cC_{\mathcal{Q}}$
the {\em $2$-category} associated to $\mathcal{Q}$. This category
is especially good in the case of a strongly regular $\mathcal{Q}$,
as follows from the following statement:

\begin{proposition}\label{prop62}
Assume $\mathcal{Q}$ is a strongly regular two-sided cell in $\mathcal{C}$.
Then $\mathcal{Q}$ remains a two-sided cell for $\cC_{\mathcal{Q}}$.
\end{proposition}

\begin{proof}
Let $\mathrm{F}\in \mathcal{Q}$. Denote by $\mathrm{G}$ the unique
self-adjoint element in the right cell $\mathcal{R}$ of $\mathrm{F}$. 
The action of $\mathrm{G}$ on the cell $2$-representation 
$\mathbf{C}_{\mathcal{R}}$ is nonzero and hence
$\mathrm{G}\neq 0$, when restricted to $\mathbf{C}_{\mathcal{R}}$. 

Further, by Proposition~\ref{lem7}\eqref{lem7.02}, 
$\mathrm{G}\, L_{\mathrm{G}}$ has simple top $L_{\mathrm{G}}$.
Using Proposition~\ref{lem7}\eqref{lem7.02} again, we thus get 
$\mathrm{F}\circ \mathrm{G}\, L_{\mathrm{G}}\neq 0$, implying
$\mathrm{F}\circ \mathrm{G}\neq 0$, when restricted 
to $\mathbf{C}_{\mathcal{R}}$. But the restriction of
$\mathrm{F}\circ \mathrm{G}$ decomposes into a direct sum of some
$\mathrm{H}\in \mathcal{Q}$, which are in the same right cell
as $\mathrm{G}$ and in the same left cell as $\mathrm{F}$.
Since $\mathcal{Q}$ is strongly regular, the only element satisfying
both conditions is $\mathrm{F}$. This implies that, when restricted
to $\mathbf{C}_{\mathcal{R}}$, $\mathrm{F}$ occurs as a
direct summand of $\mathrm{F}\circ \mathrm{G}$, which yields
$\mathrm{F}\geq_R \mathrm{G}$ in $\cC_{\mathcal{Q}}$.

Now consider the functor $\mathrm{F}^*\circ \mathrm{F}$.
Since $\mathrm{F}\neq 0$, when restricted to $\mathbf{C}_{\mathcal{R}}$,
by adjunction we have $\mathrm{F}^*\circ \mathrm{F}\neq 0$, when 
restricted to $\mathbf{C}_{\mathcal{R}}$, as well. The functor
$\mathrm{F}^*\circ \mathrm{F}$ decomposes into a direct sum of
functors from $\mathcal{R}\cap\mathcal{R}^*=\{\mathrm{G}\}$. 
This implies $\mathrm{G}\geq_R \mathrm{F}$ in $\cC_{\mathcal{Q}}$
and hence $\mathcal{R}$ remains a right cell in $\cC_{\mathcal{Q}}$.
Using $*$ we get that all left cells in $\mathcal{Q}$ remain
left cells in $\cC_{\mathcal{Q}}$. Now the claim of the proposition 
follows from Proposition~\ref{prop32}\eqref{prop32.2}.
\end{proof}

The important property of $\cC_{\mathcal{Q}}$ is that 
for strongly regular right cells the corresponding cell
$2$-representations can be studied over $\cC_{\mathcal{Q}}$:

\begin{corollary}\label{cor63}
Let $\mathcal{Q}$ be a strongly regular two-sided cell of
$\mathcal{C}$ and $\mathcal{R}$ be a right cell of $\mathcal{Q}$.
Then the restriction of the cell $2$-representation
$\mathbf{C}_{\mathcal{R}}$ from $\cC$ to $\cC_{\mathcal{Q}}$
gives the corresponding cell $2$-representation for $\cC_{\mathcal{Q}}$.
\end{corollary}

\begin{proof}
Let $\mathtt{i}\in\mathcal{C}$ be such that 
$\mathcal{R}\cap \mathcal{C}_{\mathtt{i},\mathtt{i}}\neq \varnothing$.
Denote by $\mathbf{C}^{\mathcal{Q}}_{\mathcal{R}}$ the
cell $2$-representation of $\cC_{\mathcal{Q}}$
associated to $\mathcal{R}$. We will use the upper index
$\mathcal{Q}$ for elements of this $2$-representation.
Consider $\mathbf{C}_{\mathcal{R}}$ as a $2$-representation 
of $\cC_{\mathcal{Q}}$ by restriction.
By Theorem~\ref{prop21}, we have the morphism of $2$-representations
$\Psi:=\Psi^{L_{\mathrm{G}_{\mathcal{R}}}}:
\mathbf{C}^{\mathcal{Q}}_{\mathcal{R}}\to
\mathbf{C}_{\mathcal{R}}$ 
sending $L_{\mathrm{G}_{\mathcal{R}}}^{\mathcal{Q}}$ to
$L_{\mathrm{G}_{\mathcal{R}}}$.

Let $\mathtt{j}\in\mathcal{C}$ and 
$\mathrm{F}\in \mathcal{R}\cap \mathcal{C}_{\mathtt{i},\mathtt{j}}$.
By Proposition~\ref{lem7}\eqref{lem7.02}, the morphism
$\Psi$ sends the indecomposable projective module  
$\mathrm{F}\, L_{\mathrm{G}_{\mathcal{R}}}^{\mathcal{Q}}$
of $\mathbf{C}^{\mathcal{Q}}_{\mathcal{R}}(\mathtt{j})$ to the 
indecomposable projective module $\mathrm{F}\, L_{\mathrm{G}_{\mathcal{R}}}$
in $\mathbf{C}_{\mathcal{R}}(\mathtt{j})$. As mentioned after 
Proposition~\ref{lem7}, we have that $2$-morphisms in $\cC_{\mathcal{Q}}$
surject onto the homomorphisms between indecomposable
projective modules both in $\mathbf{C}_{\mathcal{R}}^{\mathcal{Q}}(\mathtt{j})$
and $\mathbf{C}_{\mathcal{R}}(\mathtt{j})$.

To prove the claim it is left to show that $\Psi$ is
injective, when restricted to indecomposable projective modules
in $\mathbf{C}^{\mathcal{Q}}_{\mathcal{R}}(\mathtt{j})$. For this it
is enough to show that the Cartan matrices of 
$\mathbf{C}^{\mathcal{Q}}_{\mathcal{R}}(\mathtt{j})$ and
$\mathbf{C}_{\mathcal{R}}(\mathtt{j})$ coincide. For indecomposable
$\mathrm{F}$ and $\mathrm{H}$ in $\mathcal{R}\cap
\mathcal{C}_{\mathtt{i},\mathtt{j}}$, using adjunction, we have
\begin{equation}\label{eq773}
\mathrm{Hom}_{\overline{\ccC}(\mathtt{i},\mathtt{j})}
(\mathrm{F}\,L_{\mathrm{G}_{\mathcal{R}}},
\mathrm{H}\,L_{\mathrm{G}_{\mathcal{R}}})=
\mathrm{Hom}_{\overline{\ccC}(\mathtt{i},\mathtt{i})}
(\mathrm{H}^*\circ\mathrm{F}\,L_{\mathrm{G}_{\mathcal{R}}},
L_{\mathrm{G}_{\mathcal{R}}}). 
\end{equation}
and similarly for $\mathbf{C}^{\mathcal{Q}}_{\mathcal{R}}$.
The dimension of the right hand side of \eqref{eq773} equals the 
multiplicity of $\mathrm{G}_{\mathcal{R}}$ as a direct summand of
$\mathrm{H}^*\circ\mathrm{F}$. Since this multiplicity is the
same for $\mathbf{C}^{\mathcal{Q}}_{\mathcal{R}}$ and
$\mathbf{C}_{\mathcal{R}}$, the claim follows.
\end{proof}

\subsection{Cell $2$-representations for strongly regular cells}\label{s5.3}

In this section we fix a strongly regular two-sided cell $\mathcal{Q}$ 
in $\mathcal{C}$. We would like to understand combinatorics of
the cell $2$-representation $\mathbf{C}_{\mathcal{R}}$ for
a right cell $\mathcal{R}\subset \mathcal{Q}$. By the previous
subsection, for this it is enough to assume that $\cC=\cC_{\mathcal{Q}}$.
We work under this assumption in the rest of this subsection
and consider the direct sum $\mathbf{C}$ of all $\mathbf{C}_{\mathcal{R}}$,
where $\mathcal{R}$ runs through the set of all right cells
in $\mathcal{Q}$. To simplify our notation, by 
$\mathrm{Hom}_{\mathbf{C}}$ we denote the homomorphism space in 
an appropriate module category $\mathbf{C}(\mathtt{i})$.

\begin{proposition}\label{prop51}
Let $\mathcal{Q}$ be as above and $\mathrm{F},\mathrm{H}\in \mathcal{Q}$. 
\begin{enumerate}[$($a$)$]
\item\label{prop51.1} For some $m_{\mathrm{F},\mathrm{H}}\in\{0,1,2,\dots\}$
we have $\mathrm{H}^*\circ \mathrm{F}\cong m_{\mathrm{F},\mathrm{H}}
\mathrm{G}$, where $\{\mathrm{G}\}=\mathcal{L}_{\mathrm{H}^*}\cap
\mathcal{R}_{\mathrm{F}}$; moreover, $m_{\mathrm{F},\mathrm{F}}\neq 0$.
\item\label{prop51.2} If $\mathrm{F}\sim_R\mathrm{H}$, then
$m_{\mathrm{F},\mathrm{H}}= m_{\mathrm{H},\mathrm{F}}$.
\item\label{prop51.3} If $\mathrm{H}=\mathrm{H}^*$ and
$\mathrm{F}\sim_R \mathrm{H}$, then $m_{\mathrm{F},\mathrm{F}}=\dim
\mathrm{End}_{\mathbf{C}}(\mathrm{F}\, L_{\mathrm{H}})$. 
\item\label{prop51.4} If $\mathrm{H}=\mathrm{H}^*$ and
$\mathrm{F}\sim_R \mathrm{H}$, then
$\mathrm{F}\circ \mathrm{H}\cong 
m_{\mathrm{H},\mathrm{H}}\mathrm{F}$
and $\mathrm{H}\circ \mathrm{F}^*\cong 
m_{\mathrm{H},\mathrm{H}}\mathrm{F}^*$
\item\label{prop51.5} If $\mathrm{H}=\mathrm{H}^*$ and
$\mathrm{H}\sim_L\mathrm{F}$, then 
$m_{\mathrm{H},\mathrm{H}}=\dim 
\mathrm{Hom}_{\mathbf{C}}(P_{\mathrm{F}},\mathrm{H}\,L_{\mathrm{F}})$.
\item\label{prop51.6} Assume $\mathrm{G}\in \mathcal{Q}$ and
$\mathrm{H}=\mathrm{H}^*$, $\mathrm{G}=\mathrm{G}^*$,
$\mathrm{H}\sim_L\mathrm{F}$ and $\mathrm{G}\sim_R\mathrm{F}$. Then
\begin{displaymath}
m_{\mathrm{F},\mathrm{F}}m_{\mathrm{G},\mathrm{G}}=
m_{\mathrm{F}^*,\mathrm{F}^*}m_{\mathrm{H},\mathrm{H}}.
\end{displaymath}
\end{enumerate}
\end{proposition}

\begin{proof}
By our assumptions, every indecomposable direct summand of 
$\mathrm{H}^*\circ \mathrm{F}$ belongs to the right cell of 
$\mathrm{F}$ and the left cell of $\mathrm{H}^*$, hence is
isomorphic to $\mathrm{G}$. Note that $\mathrm{F}^*\circ \mathrm{F}$
is nonzero by adjunction since 
$\mathrm{F}\, L_{\mathrm{G}_{\mathcal{R}_{\mathrm{F}}}}
\neq 0$. This implies claim 
\eqref{prop51.1} and claim \eqref{prop51.3} follows from
Proposition~\ref{lem7}\eqref{lem7.02} using adjunction.

If $\mathrm{F}\sim_R\mathrm{H}$, then 
$\mathrm{H}^*\circ \mathrm{F}\cong m_{\mathrm{F},\mathrm{H}}
\mathrm{G}_{\mathcal{R}_{\mathrm{F}}}$ by \eqref{prop51.1}.
By Proposition~\ref{prop24}\eqref{prop24.1}, the functor
$\mathrm{G}_{\mathcal{R}_{\mathrm{F}}}$ is self-adjoint. Hence
$\mathrm{H}^*\circ \mathrm{F}$ is self-adjoint, which implies
claim \eqref{prop51.2}.

Set $m=m_{\mathrm{H},\mathrm{H}}$. Similarly to the proof of 
claim \eqref{prop51.1}, we have $\mathrm{F}\circ \mathrm{H}\cong k\mathrm{F}$
for some $k\in \{1,2,\dots\}$. Using associativity, we obtain
\begin{multline*}
k^2\mathrm{F}=k(\mathrm{F}\circ \mathrm{H})
=(k\mathrm{F})\circ \mathrm{H}
=(\mathrm{F}\circ \mathrm{H})\circ \mathrm{H}
=\\=\mathrm{F}\circ (\mathrm{H}\circ \mathrm{H})
=\mathrm{F}\circ (m\mathrm{H})
=m(\mathrm{F}\circ \mathrm{H})
=mk\mathrm{F}.
\end{multline*}
This implies claim \eqref{prop51.4} and claim \eqref{prop51.5}
follows by adjunction. 

Claim \eqref{prop51.6} follows from the following computation:
\begin{multline*}
m_{\mathrm{F},\mathrm{F}}m_{\mathrm{G},\mathrm{G}}\mathrm{F}
\overset{\eqref{prop51.4}}{=}
m_{\mathrm{F},\mathrm{F}}(\mathrm{F}\circ \mathrm{G})=
\mathrm{F}\circ (m_{\mathrm{F},\mathrm{F}}\mathrm{G})
\overset{\eqref{prop51.1}}{=}
\mathrm{F}\circ (\mathrm{F}^*\circ \mathrm{F})=\\
=(\mathrm{F}\circ \mathrm{F}^*)\circ \mathrm{F}
\overset{\eqref{prop51.1}}{=}
(m_{\mathrm{F}^*,\mathrm{F}^*}\mathrm{H})\circ \mathrm{F}=
m_{\mathrm{F}^*,\mathrm{F}^*}(\mathrm{H}\circ \mathrm{F})
\overset{\eqref{prop51.4}}{=}
m_{\mathrm{F}^*,\mathrm{F}^*}m_{\mathrm{H},\mathrm{H}}\mathrm{F}.
\end{multline*}
This completes the proof.
\end{proof}

As a corollary we obtain that the Cartan matrix of the cell
$2$-representation $\mathbf{C}_{\mathcal{R}}$ is symmetric.

\begin{corollary}\label{cor52}
Assume that $\mathcal{R}$ is a strongly regular right cell.  
Then for any $\mathrm{F},\mathrm{H}\in  \mathcal{R}$ we have 
$[\mathrm{F}L_{\mathrm{G}_{\mathcal{R}}}:L_{\mathrm{H}}]=
[\mathrm{H}L_{\mathrm{G}_{\mathcal{R}}}:L_{\mathrm{F}}]$.
\end{corollary}

\begin{proof}
We have 
\begin{displaymath}
[\mathrm{F}L_{\mathrm{G}_{\mathcal{R}}}:L_{\mathrm{H}}]=
\dim\mathrm{Hom}_{\mathbf{C}_{\mathcal{R}}}
(\mathrm{H}L_{\mathrm{G}_{\mathcal{R}}},
\mathrm{F}L_{\mathrm{G}_{\mathcal{R}}}).
\end{displaymath}
Using adjunction and Proposition~\ref{prop51}\eqref{prop51.1},
the latter equals $m_{\mathrm{H},\mathrm{F}}$. Now the
claim follows from Proposition~\ref{prop51}\eqref{prop51.2}.
\end{proof}

\begin{corollary}\label{cor53}
Let $\mathrm{F},\mathrm{H}\in\mathcal{Q}$ be such that 
$\mathrm{H}=\mathrm{H}^*$ and $\mathrm{F}\sim_L \mathrm{H}$.
Then the module $P_{\mathrm{F}}$ is a direct summand of 
$\mathrm{H}\, L_{\mathrm{F}}$ and 
$m_{\mathrm{F},\mathrm{F}}\leq m_{\mathrm{H},\mathrm{H}}$.
\end{corollary}

\begin{proof}
From  Proposition~\ref{prop51} we have:
\begin{displaymath}
m_{\mathrm{F},\mathrm{F}}=\dim\mathrm{End}_{\mathbf{C}}(P_{\mathrm{F}}),\quad
m_{\mathrm{H},\mathrm{H}}=\dim\mathrm{Hom}_{\mathbf{C}}(P_{\mathrm{F}},
\mathrm{H}\, L_{\mathrm{F}}).
\end{displaymath}
Hence to prove the corollary we just need to show that
$P_{\mathrm{F}}$ is a direct summand of $\mathrm{H}\, L_{\mathrm{F}}$.

By Proposition~\ref{prop51}, the module 
$\mathrm{F}\circ \mathrm{F}^*\, L_{\mathrm{F}}$
decomposes into a direct sum of $m_{\mathrm{F}^*,\mathrm{F}^*}$ copies
of the module $\mathrm{H}\, L_{\mathrm{F}}$. Hence it is enough
to show that $P_{\mathrm{F}}$ is a direct summand of 
$\mathrm{F}\circ \mathrm{F}^*\, L_{\mathrm{F}}$.

Let $\mathcal{R}$ be the right cell of $\mathrm{F}$.
We know that $\mathrm{F}^*\, L_{\mathrm{F}}\neq 0$.
Using adjunction and  Lemma~\ref{lem11}, we obtain that 
every simple quotient of 
$\mathrm{F}^*\, L_{\mathrm{F}}\neq 0$ is isomorphic to
$L_{\mathrm{G}_{\mathcal{R}}}$. Hence $\mathrm{F}^*\, L_{\mathrm{F}}$
surjects onto $L_{\mathrm{G}_{\mathcal{R}}}$ and, applying
$\mathrm{F}$, we have that $\mathrm{F}\circ \mathrm{F}^*\, L_{\mathrm{F}}$
surjects onto $P_{\mathrm{F}}$. Now the claim follows from 
projectivity of $P_{\mathrm{F}}$.
\end{proof}

\begin{corollary}\label{cor55}
For every $\mathrm{F}\in \mathcal{Q}$ the projective module 
$P_{\mathrm{F}}$ is injective.
\end{corollary}

\begin{proof}
Let $\mathcal{R}$ be the right cell of $\mathrm{F}$.
Since the functorial actions of $\mathrm{F}$ and $\mathrm{F}^*$ 
on $\mathbf{C}_{\mathcal{R}}$ are biadjoint, they preserve both
the additive category of projective modules and the additive
category of injective modules. Now take any injective module
$I$ and let $L_{\mathrm{H}}$ be some simple occurring in its
top. Applying $\mathrm{H}^*$ we get an injective module such that
$L_{\mathrm{G}_{\mathcal{R}}}$ occurs in its top. Applying now
$\mathrm{F}$ we get an injective module in which the projective
module $P_{\mathrm{F}}\cong\mathrm{F}\,L_{\mathrm{G}_{\mathcal{R}}}$
is a quotient. Hence $P_{\mathrm{F}}$ splits off as a direct summand
in this module and thus is injective. This completes the proof.
\end{proof}

\begin{corollary}\label{cor56}
Let $\mathrm{F},\mathrm{H}\in\mathcal{Q}$
and $\mathcal{R}$ be the right cell of $\mathrm{F}$.
\begin{enumerate}[$($a$)$]
\item\label{cor56.1} We have
$\mathrm{F}^*\, L_{\mathrm{F}}\cong P_{\mathrm{G}_{\mathcal{R}}}$.
\item\label{cor56.2} The module $\mathrm{H}\, L_{\mathrm{F}}$ 
is either zero or both projective and injective.
\end{enumerate}
\end{corollary}

\begin{proof}
Similarly to the proof of Corollary~\ref{cor53} one shows that 
the module $P_{\mathrm{G}_{\mathcal{R}}}$
is a direct summand of $\mathrm{F}^*\, L_{\mathrm{F}}$, so to prove
claim \eqref{cor56.1} we have to show that $\mathrm{F}^*\, L_{\mathrm{F}}$
is indecomposable. We will show that $\mathrm{F}^*\, L_{\mathrm{F}}$
has simple socle. Since $\mathrm{F}$ annihilates all simple
modules in $\mathbf{C}_{\mathcal{R}}$ but $L_{\mathrm{G}_{\mathcal{R}}}$,
using adjunction it follows that every simple submodule in
the socle of $\mathrm{F}^*\, L_{\mathrm{F}}$ is isomorphic to
$L_{\mathrm{G}_{\mathcal{R}}}$. On the other hand, using adjunction
and Proposition~\ref{lem7}\eqref{lem7.02} we obtain that the 
homomorphism space from $L_{\mathrm{G}_{\mathcal{R}}}$
to $\mathrm{F}^*\, L_{\mathrm{F}}$ is one-dimensional.
This means that $\mathrm{F}^*\, L_{\mathrm{F}}$ has simple socle
and proves claim \eqref{cor56.1}.

Assume that $\mathrm{H}\, L_{\mathrm{F}}\neq 0$. Then, by
Lemma~\ref{lem11}, we have $\mathrm{F}^*\sim_R \mathrm{H}$
(since $\mathcal{Q}$ is strongly regular). Let $\mathrm{G}\in\mathcal{R}$ 
be such that  $\mathrm{G}\sim_L \mathrm{H}$. Then, by 
Proposition~\ref{prop51}\eqref{prop51.1}, we have
$\mathrm{G}\circ \mathrm{F}^*\cong m_{\mathrm{F}^*,\mathrm{G}^*} 
\mathrm{H}$. So, to prove claim \eqref{cor56.2} it is enough to
show that $m_{\mathrm{F}^*,\mathrm{G}^*}\neq 0$ and that
$\mathrm{G}\circ \mathrm{F}^*\, L_{\mathrm{F}}$ is 
both projective and injective. By claim \eqref{cor56.1}, we have
$\mathrm{F}^*\, L_{\mathrm{F}}\cong P_{\mathrm{G}_{\mathcal{R}}}$.
Since $\mathrm{G}\, L_{\mathrm{G}_{\mathcal{R}}}\neq 0$
by Proposition~\ref{lem7}\eqref{lem7.02} and $\mathrm{G}$ is exact, 
it follows that $\mathrm{G}\circ \mathrm{F}^*\, L_{\mathrm{F}}\neq 0$ 
and hence $m_{\mathrm{F}^*,\mathrm{G}^*}\neq 0$. Further,
$\mathrm{G}\, P_{\mathrm{G}_{\mathcal{R}}}$ is projective as
$P_{\mathrm{G}_{\mathcal{R}}}$ is projective and
$\mathrm{G}$ is biadjoint to $\mathrm{G}^*$.
Finally, $\mathrm{G}\, P_{\mathrm{G}_{\mathcal{R}}}$ is injective
by Corollary~\ref{cor55}. Claim \eqref{cor56.2} follows and the
proof is complete.
\end{proof}

\begin{corollary}\label{cor57}
Let $\mathrm{F},\mathrm{H}\in\mathcal{Q}$ be such that 
$\mathrm{H}=\mathrm{H}^*$ and $\mathrm{F}\sim_L \mathrm{H}$.
Then $m_{\mathrm{F},\mathrm{F}}\vert m_{\mathrm{H},\mathrm{H}}$.
\end{corollary}

\begin{proof}
Let $\mathcal{R}$ be the right cell of $\mathrm{F}$.
By Lemma~\ref{lem11}, $\mathrm{H}$ annihilates all simples 
of $\mathbf{C}_{\mathcal{R}}$ but $L_{\mathrm{F}}$. This
and Corollary~\ref{cor56}\eqref{cor56.2} imply that
$\mathrm{H}\, L_{\mathrm{F}}=k P_{\mathrm{F}}$ for some
$k\in\mathbb{N}$. On the one hand, using Propositions~\ref{lem7}
and \ref{prop51} we have
\begin{equation}\label{eq51}
(\mathrm{F}^*\circ\mathrm{H})\, L_{\mathrm{F}}=
k\mathrm{F}^* P_{\mathrm{F}}= 
k(\mathrm{F}^*\circ \mathrm{F}) L_{\mathrm{G}_{\mathcal{R}}}=
km_{\mathrm{F},\mathrm{F}}\mathrm{G}_{\mathcal{R}}\,
L_{\mathrm{G}_{\mathcal{R}}}=
km_{\mathrm{F},\mathrm{F}}\, P_{\mathrm{G}_{\mathcal{R}}}.
\end{equation}
On the other hand, we have $\mathrm{F}^*\sim_R \mathrm{H}$
and thus, using Proposition~\ref{prop51}\eqref{prop51.4}
and Corollary~\ref{cor56}\eqref{cor56.1}, we have:
\begin{equation}\label{eq52}
(\mathrm{F}^*\circ\mathrm{H})\, L_{\mathrm{F}}=
m_{\mathrm{H},\mathrm{H}}\mathrm{F}^*\, L_{\mathrm{F}}=
m_{\mathrm{H},\mathrm{H}}P_{\mathrm{G}_{\mathcal{R}}}.
\end{equation}
The claim follows comparing \eqref{eq51} and \eqref{eq52}.
\end{proof}

\section{Cyclic and simple $2$-representations of fiat categories}\label{s4}

\subsection{Cyclic $2$-representations}\label{s4.1}

Let $\cC$ be a fiat category, $\mathbf{M}$ a $2$-representation
of $\cC$, $\mathtt{i}\in\cC$ and $M\in \mathbf{M}(\mathtt{i})$.
We will say that $M$ {\em generates} $\mathbf{M}$ if for 
any $\mathtt{j}\in\cC$ and $X,Y\in \mathbf{M}(\mathtt{j})$ there are 
$\mathrm{F},\mathrm{G}\in\hat{\cC}(\mathtt{i},\mathtt{j})$ 
such that $\mathrm{F}\, M\cong X$, $\mathrm{G}\, M\cong Y$ and the
evaluation map 
$\mathrm{Hom}_{\hat{\ccC}(\mathtt{i},\mathtt{j})}(\mathrm{F},\mathrm{G})
\rightarrow \mathrm{Hom}_{\mathbf{M}(\mathtt{j})}
(\mathrm{F}\, M,\mathrm{G}\, M)$ is surjective.
The $2$-representation $\mathbf{M}$ is called {\em cyclic}
provided that there exists $\mathtt{i}\in\cC$ and 
$M\in \mathbf{M}(\mathtt{i})$ such that $M$ generates $\mathbf{M}$.
Examples of cyclic $2$-representations of $\cC$ are given by the
following:

\begin{proposition}\label{prop41}
\begin{enumerate}[$($a$)$]
\item\label{prop41.1} 

For any $\mathtt{i}\in\cC$ the $2$-representation
$\mathbf{P}_{\mathtt{i}}$ is cyclic and generated by 
$P_{\mathbbm{1}_{\mathtt{i}}}$.
\item\label{prop41.2} For any right cell $\mathcal{R}$ of $\mathcal{C}$
the cell $2$-representation $\mathbf{C}_{\mathcal{R}}$ is cyclic and 
generated by $L_{\mathrm{G}_{\mathcal{R}}}$.
\end{enumerate} 
\end{proposition}

\begin{proof}
Let $\mathtt{j}\in\cC$, $X,Y\in \mathbf{P}_{\mathtt{i}}(\mathtt{j})$ and
$f:X\to Y$. Taking some projective presentations of $X$ and $Y$ 
yields the following commutative diagram with exact rows:
\begin{equation}\label{eq42}
\xymatrix{
X_1\ar[rr]^{h}\ar[d]_{f''}&&X_0\ar@{->>}[rr]\ar[d]^{f'}&&X\ar[d]^f\\
Y_1\ar[rr]^{g}&&Y_0\ar@{->>}[rr]&&Y\\
}
\end{equation}
Now $X_1,X_0,Y_1,Y_0$ are projective in 
$\mathbf{P}_{\mathtt{i}}(\mathtt{j})$ and we may assume that 
$X_1=\mathrm{F}_1\,P_{\mathbbm{1}_{\mathtt{i}}}$, 
$X_0=\mathrm{F}_0\,P_{\mathbbm{1}_{\mathtt{i}}}$, 
$Y_1=\mathrm{G}_1\,P_{\mathbbm{1}_{\mathtt{i}}}$ and 
$Y_0=\mathrm{G}_0\,P_{\mathbbm{1}_{\mathtt{i}}}$ for some
$\mathrm{F}_1,\mathrm{F}_0,\mathrm{G}_1,\mathrm{G}_0\in 
\cC(\mathtt{i},\mathtt{j})$. From the definition of 
$\mathbf{P}_{\mathtt{i}}$ we then obtain that $g$, $h$, $f'$ and
$f''$ are given by $2$-morphisms between the corresponding
$1$-morphisms (which we denote by the same symbols). 

It follows that $X$ equals to the image of
$P_{\mathbbm{1}_{\mathtt{i}}}$ under 
$\mathrm{H}_1:=\mathrm{Coker}
(\mathrm{F}_1\overset{h}{\rightarrow}\mathrm{F}_0)\in
\hat{\cC}(\mathtt{i},\mathtt{j})$. Similarly, $Y$ equals to the 
image of  $P_{\mathbbm{1}_{\mathtt{i}}}$ under 
$\mathrm{H}_2:=\mathrm{Coker}(
\mathrm{G}_1\overset{g}{\rightarrow}\mathrm{G}_0)\in
\hat{\cC}(\mathtt{i},\mathtt{j})$. Finally,
$f$ is induced by the diagram
\begin{displaymath}
\xymatrix{
\mathrm{F}_1\ar[rr]^{h}\ar[d]_{f''}&&\mathrm{F}_0\ar[d]^{f'}\\
\mathrm{G}_1\ar[rr]^{g}&&\mathrm{G}_0.\\
}
\end{displaymath}
Claim~\eqref{prop41.1} follows.

To prove claim~\eqref{prop41.2} we view every
$\mathbf{C}_{\mathcal{R}}(\mathtt{j})$ as the corresponding full
subcategory of $\mathbf{P}_{\mathtt{i}}(\mathtt{j})$.
Let $X,Y\in \mathbf{C}_{\mathcal{R}}(\mathtt{j})$ 
and $f:X\to Y$. From the proof of claim~\eqref{prop41.1} we have 
the commutative  diagram \eqref{eq42} as described above.
Our proof of claim~\eqref{prop41.2} will proceed by
certain manipulations of this diagram. Denote by
$\mathcal{I}$ the ideal of $\mathcal{C}$ with respect
to $\leq_R$ generated by $\mathcal{R}$ and 
set $\mathcal{I}':=\mathcal{I}\setminus  \mathcal{R}$.

To start with, we modify the left column of \eqref{eq42}. 
Let $X'_1$ and $Y'_1$ denote the trace of all projective modules 
of the form  $P_{\mathrm{G}}$, $\mathrm{G}\not\in\mathcal{I}'$, 
in $X_1$ and $Y_1$, respectively. Consider some minimal projective 
covers $\hat{X}_1\tto X'_1$ and  $\hat{Y}_1\tto X'_1$  of $X'_1$ and 
$Y'_1$, respectively. Let $\mathrm{can}:\hat{X}_1\tto X'_1
\hookrightarrow  X_1$ and  $\mathrm{can}':\hat{Y}_1\tto Y'_1
\hookrightarrow  Y_1$ denote the corresponding canonical maps
and set $\hat{h}=h\circ \mathrm{can}$ and 
$\hat{g}=g\circ \mathrm{can}'$.
Then the cokernel of both $\mathrm{can}$ and $\mathrm{can}'$
has only composition factors of the form
$L_{\mathrm{F}}$, $\mathrm{F}\in\mathcal{I}'$.
By construction, the image of $f''\circ \mathrm{can}$ is contained
in the image of $\mathrm{can}'$. Hence, using projectivity of 
$\hat{X}_1$, the map $f''$ lifts to a
map $\hat{f''}:\hat{X}_1\to \hat{Y}_1$ such that the
following diagram commutes:
\begin{equation}\label{eq43}
\xymatrix{
\hat{X}_1\ar@/^2pc/[rrrr]_{\hat{h}}\ar[d]_{\hat{f''}}
\ar@{.>}[rr]_{\mathrm{can}}&&X_1\ar@{.>}[d]^{f''}
\ar@{.>}[rr]_{{h}}
&&X_0
\ar@{->>}[rr]\ar[d]^{f'}&&X\ar[d]^f\\
\hat{Y}_1\ar@/_2pc/[rrrr]^{\hat{g}}
\ar@{.>}[rr]^{\mathrm{can}'}&&Y_1
\ar@{.>}[rr]^{{g}}&&Y_0\ar@{->>}[rr]&&Y\\
}
\end{equation}
The difference between \eqref{eq42} and \eqref{eq43} is that the rows 
of the solid part of \eqref{eq43} are no longer exact but might have 
homology in the middle. By construction, all simple subquotients of 
these homologies have the form $L_{\mathrm{F}}$, 
$\mathrm{F}\in\mathcal{I}'$. Further, all projective direct summands
appearing in \eqref{eq43} have the form $P_{\mathrm{F}}$ for
$\mathrm{F}\not\in \mathcal{I}'$.

Denote by $\tilde{X}_1$, $\tilde{X}_0$, $\tilde{Y}_1$ and
$\tilde{Y}_0$ the submodules of $\hat{X}_1$, $X_0$, $\hat{Y}_1$ and 
$Y_0$, respectively, which are uniquely defined by the following
construction: The corresponding submodules
contain all direct summands of the form $P_{\mathrm{F}}$ for
$\mathrm{F}\not\in \mathcal{R}$; and for each direct summand
of the form $P_{\mathrm{F}}$, $\mathrm{F}\in \mathcal{R}$, the 
corresponding submodules contain the submodule
$\mathrm{Ker}_{\mathrm{F}}$ of $P_{\mathrm{F}}$ as defined in 
\eqref{eq777}.  By construction and Lemma~\ref{lem16}, 
we have $\hat{h}:\tilde{X}_1\to \tilde{X}_0$,
$\hat{g}:\tilde{Y}_1\to \tilde{Y}_0$, $f':\tilde{X}_0\to \tilde{Y}_0$
and $\hat{f}'':\tilde{X}_1\to \tilde{Y}_1$. Since
$X,Y\in \mathbf{C}_{\mathcal{R}}$, the images (on diagram \eqref{eq43})
of $\tilde{X}_0$ and $\tilde{Y}_0$ in $X$ and $Y$, respectively, 
are zero. Hence, taking quotients gives the following commutative diagram:
\begin{equation}\label{eq44}
\xymatrix{
\hat{X}_1/\tilde{X}_1\ar[rr]^{\tilde{h}}\ar[d]_{\tilde{f}''}&&
X_0/\tilde{X}_0\ar@{->>}[rr]\ar[d]^{\tilde{f}'}&&
X\ar[d]^{f}\\
\hat{Y}_1/\tilde{Y}_1\ar[rr]^{\tilde{g}}&&Y_0/\tilde{Y}_0\ar@{->>}[rr]&&
Y,\\
}
\end{equation}
where $\tilde{h}$, $\tilde{g}$, $\tilde{f}''$ and
$\tilde{f}'$ denote the corresponding induced maps.

By our construction of  \eqref{eq44} and definition 
of $\mathbf{C}_{\mathcal{R}}$, all indecomposable modules appearing 
in the left square of \eqref{eq44} are projective
in $\mathbf{C}_{\mathcal{R}}(\mathtt{j})$ (and hence,  by definition of 
$\mathbf{C}_{\mathcal{R}}$, have the form
$\mathrm{F}\,L_{\mathrm{G}_{\mathcal{R}}}$ for some
$\mathrm{F}\in \mathcal{R}$). Moreover, all simples of the 
form $L_{\mathrm{F}}$, $\mathrm{F}\in\mathcal{I}'$, become zero in 
$\mathbf{C}_{\mathcal{R}}(\mathtt{j})$ 
(since $\mathbf{C}_{\mathcal{R}}(\mathtt{j})$ is defined as
a Serre subquotient and simples  $L_{\mathrm{F}}$, 
$\mathrm{F}\in\mathcal{I}'$, belong to the kernel).
This implies that both rows of \eqref{eq44} are exact in
$\mathbf{C}_{\mathcal{R}}(\mathtt{j})$.
As mentioned in the proof of claim \eqref{prop41.1},
the maps $g,h,f'$ and $f''$ on diagram \eqref{eq42} are 
given by $2$-morphisms in $\cC$. Similarly, the maps $\hat{g}$, 
$\hat{h}$ and $\hat{f''}$ on diagram \eqref{eq43} are given by 
$2$-morphisms in $\cC$ as well. By construction of \eqref{eq44},
the maps $\tilde{h}$, $\tilde{g}$, $\tilde{f}''$ and
$\tilde{f}'$ are induced by $\hat{h}$, $\hat{g}$, $\hat{f''}$ and
$f'$, respectively. Now the proof of claim \eqref{prop41.2}
is completed similarly to the proof of claim \eqref{prop41.1}.
\end{proof}

\subsection{Simple $2$-representations}\label{s4.2}

A (nontrivial) $2$-representation $\mathbf{M}$ of $\cC$ is 
called {\em quasi-simple} provided that it is cyclic and generated 
by a simple module. From Proposition~\ref{prop41}\eqref{prop41.2} 
it follows that every cell  $2$-representation is quasi-simple. 
A (nontrivial) $2$-representation $\mathbf{M}$ of $\cC$ is 
called {\em strongly simple} provided that it is cyclic and 
generated by any simple module. It turns out that for strongly 
regular right cells strong simplicity of cell $2$-representations
behaves well with respect to restrictions.

\begin{proposition}\label{prop67}
Let $\mathcal{Q}$ be a strongly regular two-sided cell
and $\mathcal{R}$ a right cell in $\mathcal{Q}$. Then
the cell $2$-representation $\mathbf{C}_{\mathcal{R}}$
of $\cC$ is strongly simple if and only if its restriction
to $\cC_{\mathcal{Q}}$ is strongly simple.
\end{proposition}

To prove this we will need the following general lemma:

\begin{lemma}\label{lem65}
Let $\mathcal{Q}$ be two-sided cell  and
$\mathbf{M}$ a $2$-representations of $\cC$.
Let $\mathrm{H}\in\hat{\cC}$ be such that for any
$\mathrm{F}\in\mathcal{C}$ the inequality
$\mathrm{Hom}_{\hat{\ccC}}(\mathrm{F},\mathrm{H})\neq 0$
implies $\mathrm{F}<_{LR}\mathcal{Q}$. Then 
for any $\mathrm{G}\in\mathcal{Q}$ the functor
$\mathbf{M}(\mathrm{G})$ annihilates the image of
$\mathbf{M}(\mathrm{H})$.
\end{lemma}

\begin{proof}
Let $\mathrm{H}=\mathrm{Coker}(\alpha)$, where
$\alpha:\mathrm{H}'\to\mathrm{H}''$ is a $2$-morphism in $\cC$.
From Lemma~\ref{lem11}, applied to an appropriate
$\mathbf{P}_{\mathtt{i}}$, it follows that 
$\mathrm{G}(\alpha)$ is surjective. This implies that
$\mathrm{G}\circ\mathrm{H}=0$ and yields the claim.
\end{proof}

\begin{proof}[Proof of Proposition~\ref{prop67}.]
Let $\mathrm{H}\in\mathcal{C}$ be such that
$\mathrm{H}<_{LR}\mathcal{Q}$. Then there exists a
$1$-mor\-phism $\mathrm{F}$ in $\cC$ and a $2$-morphism
$\alpha:\mathrm{F}\to\mathrm{H}$ in $\cC$ such that
every indecomposable direct summand of 
$\mathbf{C}_{\mathcal{R}}(\mathrm{F})$ has the form
$\mathbf{C}_{\mathcal{R}}(\mathrm{G})$ for some
$\mathrm{G}\in\mathcal{Q}$ and the cokernel of
$\mathbf{C}_{\mathcal{R}}(\alpha)$ satisfies the
condition that for any $\mathrm{K}\in\mathcal{C}$
the existence of a nonzero homomorphism from
$\mathbf{C}_{\mathcal{R}}(\mathrm{K})$ to 
$\mathrm{Coker}(\mathbf{C}_{\mathcal{R}}(\alpha))$
implies $\mathrm{K}<_{LR}\mathcal{Q}$. By Lemma~\ref{lem65},
every $1$-morphism in $\mathcal{Q}$ annihilates the
image of $\mathrm{Coker}(\mathbf{C}_{\mathcal{R}}(\alpha))$.
Since every simple in $\mathbf{C}_{\mathcal{R}}$ is not
annihilated by some $1$-morphism in $\mathcal{Q}$, we
have that the image of 
$\mathrm{Coker}(\mathbf{C}_{\mathcal{R}}(\alpha))$
is zero and hence 
$\mathrm{Coker}(\mathbf{C}_{\mathcal{R}}(\alpha))$ is
the zero functor. This means that 
$\mathbf{C}_{\mathcal{R}}(\mathrm{H})$ is a quotient
of $\mathbf{C}_{\mathcal{R}}(\mathrm{F})$, which 
implies the claim.
\end{proof}

The following theorem is our main result (and a proper
formulation of Theorem~\ref{thmmain} from Section~\ref{s0}).

\begin{theorem}[Strong simplicity of cell $2$-representations]\label{thm61} 
Let $\mathcal{Q}$ be a strongly regular two-sided cell. Assume that 
\begin{equation}\label{eq62}
\text{the function}\qquad
\begin{array}{ccc}
\mathcal{Q}&\longrightarrow & \mathbb{N}_0\\
\mathrm{F}&\mapsto& m_{\mathrm{F},\mathrm{F}}
\end{array}\qquad
\text{is constant on left cells of $\mathcal{Q}$}.
\end{equation}
Then we have:
\begin{enumerate}[$($a$)$]
\item\label{thm61.2.1} For any right cell $\mathcal{R}$
in $\mathcal{Q}$ the cell $2$-representation
$\mathbf{C}_{\mathcal{R}}$ is strongly simple.
\item\label{thm61.2.2} If $\mathcal{R}$ and $\mathcal{R}'$ 
are two right cells in $\mathcal{Q}$, then the cell $2$-representations
$\mathbf{C}_{\mathcal{R}}$ and $\mathbf{C}_{\mathcal{R}'}$
are equivalent.
\end{enumerate}
\end{theorem}

\begin{proof}
Let $\mathrm{F},\mathrm{G}\in\mathcal{R}$ and $\mathrm{H}\in\mathcal{Q}$
be such that $\mathrm{H}\in\mathcal{R}_{\mathrm{F}^*}\cap
\mathcal{L}_{\mathrm{G}}$. The module $\mathrm{H}\,L_{\mathrm{F}}$
is nonzero by Lemma~\ref{lem11} and
projective by Corollary~\ref{cor56}\eqref{cor56.2}. From 
Proposition~\ref{prop51} and Corollary~\ref{cor56}\eqref{cor56.2} 
it follows that $\mathrm{H}\,L_{\mathrm{F}}\cong kP_{\mathrm{G}}$ for some
$k\in\mathbb{N}$. Hence, by adjunction,
\begin{displaymath}
m_{\mathrm{H},\mathrm{H}}=
\dim \mathrm{End}_{\mathbf{C}_{\mathcal{R}}}(\mathrm{H}\,L_{\mathrm{F}})=
\dim \mathrm{End}_{\mathbf{C}_{\mathcal{R}}}(kP_{\mathrm{G}})=
k^2\dim \mathrm{End}_{\mathbf{C}_{\mathcal{R}}}(P_{\mathrm{G}})=
k^2m_{\mathrm{G},\mathrm{G}}. 
\end{displaymath}
On the other hand, $\mathrm{H}\sim_L\mathrm{G}$ and thus
$m_{\mathrm{H},\mathrm{H}}=m_{\mathrm{G},\mathrm{G}}$ by 
our assumption \eqref{eq62},
which implies $k=1$. This means that every 
$\mathrm{H}\in \mathcal{R}_{\mathrm{F}^*}$ maps $L_{\mathrm{F}}$ to 
an indecomposable projective module.

To prove \eqref{thm61.2.1} it is left to show that $2$-morphisms
in $\cC$ surject onto homomorphisms between indecomposable
projective modules. By adjunction, it is enough to show that for any
$\mathrm{H},\mathrm{J}\in \mathcal{R}_{\mathrm{F}^*}$ the space
of $2$-morphisms from $\mathrm{H}^*\circ\mathrm{J}$ to the
identity surjects onto homomorphisms from the projective module
$\mathrm{H}^*\circ\mathrm{J}\, L_{\mathrm{F}}$ to $L_{\mathrm{F}}$.
For the latter homomorphism space to be nonzero, 
the functor $\mathrm{H}^*\circ\mathrm{J}$ should decompose into a 
direct sum of copies of $\mathrm{K}\in \mathcal{R}_{\mathrm{F}^*}$ 
such that $\mathrm{K}\cong \mathrm{K}^*$ (see 
Proposition~\ref{prop51}\eqref{prop51.1}). By additivity, it is enough 
to show that there is a $2$-morphism from $\mathrm{K}$ to the identity 
such that its evaluation at $L_{\mathrm{F}}$ is nonzero. We have
$\mathrm{K}\, L_{\mathrm{F}}\neq 0$ by Lemma~\ref{lem11}, which implies 
that the evaluation  at $L_{\mathrm{F}}$ of the adjunction morphism 
from $\mathrm{K}\circ\mathrm{K}$ to the identity is nonzero. We have
$\mathrm{K}\circ\mathrm{K}\cong m_{\mathrm{K},\mathrm{K}}\mathrm{K}\neq 0$
by Proposition~\ref{prop51}\eqref{prop51.1}. By additivity, the
nonzero adjunction morphism restricts to a morphism from
one of the summands such that the evaluation at $L_{\mathrm{F}}$ 
remains nonzero. Claim \eqref{thm61.2.1} follows.

To prove \eqref{thm61.2.2}, consider the cell $2$-representations
$\mathbf{C}_{\mathcal{R}}$ and $\mathbf{C}_{\mathcal{R}'}$.
Without loss of generality we may assume that  $\mathcal{Q}$ is the 
unique maximal two-sided cell with respect to $\leq_{LR}$.
Let $\mathrm{G}:=\mathrm{G}_{\mathcal{R}}$ and denote by
$\mathrm{F}$ the unique element in 
$\mathcal{R}'\cap\mathcal{L}_{\mathrm{G}}$. Then
$\mathrm{G}\, L_{\mathrm{F}}\neq 0$ by Lemma~\ref{lem11}.
Moreover, from the proof of \eqref{thm61.2.1} we know that
$\mathrm{G}\, L_{\mathrm{F}}$ is an indecomposable projective 
module and hence has simple top.

Assume that $\mathtt{i}\in\cC$ is such that
$\mathrm{G}\in\cC(\mathtt{i},\mathtt{i})$.
Let $\mathrm{K}$ be a $1$-morphism in $\cC$ and
$\alpha:\mathrm{K}\to \mathrm{G}$ be a $2$-morphism
such that $\mathbf{P}_{\mathtt{i}}(\alpha)$ is a projective
presentation of $L_{\mathrm{G}}$. Denote by 
$\hat{\mathrm{G}}\in\hat{\cC}$
the cokernel of $\alpha$. 

\begin{lemma}\label{lem64}
The module $\hat{\mathrm{G}}\, L_{\mathrm{F}}$
surjects onto $L_{\mathrm{F}}$.
\end{lemma}

\begin{proof}
It is enough to prove that $\hat{\mathrm{G}}\, L_{\mathrm{F}}\neq 0$.
Since $\mathrm{G}\, L_{\mathrm{F}}$ has simple top,
it is enough to show that for any indecomposable 
$1$-morphism $\mathrm{M}$ and any $2$-morphism
$\beta:\mathrm{M}\to \mathrm{G}$ which is not an isomorphism,
the morphism $\beta_{L_{\mathrm{F}}}$ is not surjective.

The statement is obvious if $\mathrm{M}\, L_{\mathrm{F}}=0$. 
If $\mathrm{M}\, L_{\mathrm{F}}\neq 0$, we have $\mathrm{M}\leq_R
\mathrm{F}^*$ by Lemma~\ref{lem11}. Hence either 
$\mathrm{M}\sim_R\mathrm{G}$ or $\mathrm{M}<_R\mathrm{G}$.
If $\mathrm{M}=\mathrm{G}$, then $\beta$ is a radical endomorphism of
$\mathrm{G}$, hence nilpotent (as $\cC$ is a fiat category). This
means that $\beta_{L_{\mathrm{F}}}$ is nilpotent and thus is not 
surjective. If $\mathrm{M}\in\mathcal{R}\setminus\mathrm{G}$, then
$\mathrm{M}^*\not\in \mathcal{R}$ and hence
$\mathrm{M}^*L_{\mathrm{F}}=0$ by Lemma~\ref{lem11}.
By adjunction this implies that $L_{\mathrm{F}}$ does not occur
in the top of $\mathrm{M}\, L_{\mathrm{F}}$, which means that
$\beta_{L_{\mathrm{F}}}$ cannot be surjective. This implies the claim 
for all  $\mathrm{M}\in\mathcal{Q}$.

Consider now the remaining case $\mathrm{M}<_R\mathrm{G}$
and assume that $\beta_{L_{\mathrm{F}}}$ is surjective.
Let $\mathrm{M}''$ be a $1$-morphism and
$\gamma:\mathrm{M}''\to \mathrm{M}$ be a $2$-morphism such that
$\gamma$ gives the trace in $\mathrm{M}$ of all 
$1$-morphisms $\mathrm{J}$ satisfying $\mathrm{J}\not<_R\mathcal{R}$.
Denote by $\mathrm{M}'$ and $\mathrm{G}'$ the cokernels of 
$\gamma$ and $\beta\circ_1\gamma$, respectively. Then both $\mathrm{M}'$ and
$\mathrm{G}'$ are in $\hat{\cC}$. Let $\beta':\mathrm{M}'\to
\mathrm{G}'$ be the $2$-morphism induced by $\beta$. 
Any direct summand of $\mathrm{M}''$ which does not annihilate
$L_{\mathrm{F}}$ has the form $\hat{\mathrm{M}}$ for some
$\hat{\mathrm{M}}\in \mathcal{Q}$ because of our construction 
and maximality of $\mathcal{Q}$. Hence from the 
previous paragraph it follows that the map
$\beta'_{L_{\mathrm{F}}}$ is still surjective.
On the other hand, because of our construction of $\mathrm{M}''$,
an application of  Lemma~\ref{lem65} gives
$\mathrm{G}\circ \mathrm{M}'=0$ in $\hat{\cC}$.
At the same time, the nonzero module $\mathrm{G}'\, L_{\mathrm{F}}$
is a quotient of $\mathrm{G}\, L_{\mathrm{F}}$ and hence has
simple top $L_{\mathrm{F}}$. This implies 
$\mathrm{G}\circ \mathrm{G}'\, L_{\mathrm{F}}\neq 0$.
Therefore, applying $\mathrm{G}$ to the epimorphism
\begin{displaymath}
\beta'_{L_{\mathrm{F}}}:\mathrm{M}'\, L_{\mathrm{F}}\tto
\mathrm{G}'\, L_{\mathrm{F}}
\end{displaymath}
annihilates the left hand side and does not annihilate the right hand side.
This contradicts the right exactness of $\mathrm{G}$ and the claim follows.
\end{proof}

By \eqref{thm61.2.1}, any extension of $L_{\mathrm{F}}$ 
by any other simple in $\mathbf{C}_{\mathcal{R}'}$ comes from
some $2$-morphism in $\cC$. Hence this extension cannot
appear in $\hat{\mathrm{G}}\, L_{\mathrm{F}}$ by construction of
$\hat{\mathrm{G}}$. This and Lemma~\ref{lem64} imply 
$\hat{\mathrm{G}}\, L_{\mathrm{F}}\cong L_{\mathrm{F}}$.

Therefore, by Theorem~\ref{prop21},  there is a unique homomorphism 
$\Psi:\mathbf{C}_{\mathcal{R}}\to\mathbf{C}_{\mathcal{R}'}$ 
of $2$-representations, which maps $L_{\mathrm{G}}$
to $L_{\mathrm{F}}$. From claim \eqref{thm61.2.1} it follows
that $\Psi$ maps indecomposable projectives to indecomposable
projectives. Restrict $\Psi$ to $\cC_{\mathcal{Q}}$.
Then from the proof of \eqref{thm61.2.1} we have that 
for any $\mathrm{H}_1,\mathrm{H}_2\in \mathcal{R}$ we have
\begin{displaymath}
\dim \mathrm{Hom}_{\mathbf{C}_{\mathcal{R}}}
(P_{\mathrm{H}_1},P_{\mathrm{H}_2})=
\dim \mathrm{Hom}_{\mathbf{C}_{\mathcal{R}'}}
(\Psi\,P_{\mathrm{H}_1},\Psi\,P_{\mathrm{H}_2}).
\end{displaymath}
Moreover, both spaces are isomorphic to 
$\cC_{\mathcal{Q}}(\mathrm{H}_1,\mathrm{H}_2)$.
From \eqref{thm61.2.1} and construction of $\Psi$ it follows that 
$\Psi$ induces an isomorphism between 
$\mathrm{Hom}_{\mathbf{C}_{\mathcal{R}}}(P_{\mathrm{H}_1},P_{\mathrm{H}_2})$ 
and $\mathrm{Hom}_{\mathbf{C}_{\mathcal{R}'}}
(\Psi\,P_{\mathrm{H}_1},\Psi\,P_{\mathrm{H}_2})$. This means that
$\Psi$ induces an equivalence between the additive categories of
projective modules in $\mathbf{C}_{\mathcal{R}}$ and
$\mathbf{C}_{\mathcal{R}'}$. Since $\Psi$ is right exact, this implies 
that $\Psi$ is an equivalence of categories and completes the proof.
\end{proof}

\section{Examples}\label{s6}

\subsection{Projective functors on the regular block of the
category $\mathcal{O}$}\label{s6.1}

Let $\mathfrak{g}$ denote a semi-simple complex 
finite dimensional Lie algebra with a fixed 
triangular decomposition $\mathfrak{g}=
\mathfrak{n}_-\oplus\mathfrak{h}\oplus\mathfrak{n}_+$
and $\mathcal{O}_0$ the principal block of the
BGG-category $\mathcal{O}$ for $\mathfrak{g}$ (see \cite{Hu}). If $W$ denotes the Weyl group of 
$\mathfrak{g}$, then simple objects in $\mathcal{O}_0$
are simple highest weight modules $L(w)$, $w\in W$,
of highest weight $w\cdot 0\in\mathfrak{h}^*$.
Denote by $P(w)$ the indecomposable projective cover of
$L(w)$ and by $\Delta(w)$ the corresponding Verma
module.

Let $\cS=\cS_{\mathfrak{g}}$ denote the (strict) $2$-category defined 
as follows: it has one object $\mathtt{i}$ (which we identify with
$\mathcal{O}_0$); its $1$-morphisms are projective functors on 
$\mathcal{O}_0$, that is functors isomorphic to direct 
summands of tensoring with finite dimensional
$\mathfrak{g}$-modules (see \cite{BG}); and its 
$2$-morphisms are natural transformations of functors. 
For $w\in W$ denote by $\theta_w$ the unique
(up to isomorphism) indecomposable projective
functor on $\mathcal{O}_0$ sending $P(e)$
to $P(w)$. Then $\{\theta_w:w\in W\}$ is a complete
and irredundant list of representatives of isomorphism
classes of indecomposable projective functors.
Since $\mathcal{O}_0$ is equivalent to the category of 
modules over a finite-dimensional associative algebra, 
all spaces of $2$-morphisms in $\cS$ are finite dimensional. 
From \cite{BG} we also have that $\cS$ is 
stable under taking adjoint functors.
It follows that $\cS$ is a fiat category.
The split Grothendieck ring $[\cS]_{\oplus}$ of $\cS$ 
is isomorphic to the integral group ring $\mathbb{Z}W$ 
such that the basis $\{[\theta_w]:w\in W\}$ of $[\cS]_{\oplus}$
corresponds to the Kazhdan-Lusztig basis of $\mathbb{Z}W$.
We refer the reader to \cite{Ma} for an overview and 
more details on this category.

Left and right cells of $\cS$ are given by the
Kazhdan-Lusztig combinatorics for $W$ (see \cite{KaLu}) 
and correspond to Kazhdan-Lusztig left and right cells in $W$, 
respectively. Namely, for $x,y\in W$ the functors
$\theta_x$ and $\theta_y$ belong to the same 
left (right or two-sided) cell as defined in
Subsection~\ref{s3.1} if and only if $x$ and $y$
belong to the same Kazhdan-Lusztig left (right or two-sided)
cell, respectively. This is an immediate consequence
of the multiplication formula for elements of the
Kazhdan-Lusztig basis (see \cite{KaLu}). In particular,
from \cite{Lu} it follows that all cells for
$\cS$ are regular. If $\mathfrak{g}\cong\mathfrak{sl}_n$,
then $W$ is isomorphic to the symmetric group $S_n$.
Robinson-Schensted correspondence associates to every
$w\in S_n$ a pair $(\alpha(w),\beta(w))$ of standard Young
tableaux of the same shape (see \cite[Section~3.1]{Sa}).
Elements $x,y\in S_n$ belong to the same Kazhdan-Lusztig
right or left cell if and only if $\alpha(x)=\alpha(y)$
and $\beta(x)=\beta(y)$, respectively (see \cite{KaLu}).
It follows that in the case $\mathfrak{g}\cong\mathfrak{sl}_n$
all cells for $\cS$ are strongly regular.

The $2$-category $\cS$ comes along with the {\em defining
$2$-representation}, that is the natural action of $\cS$
on $\mathcal{O}_0$. Various $2$-representations of
$\cS$ were constructed, as subquotients of the defining
representation, in \cite{KMS} and \cite{MS2} (see 
also \cite{Ma} for a more detailed overview).
In particular,  in \cite{MS2} for every Kazhdan-Lusztig right cell 
$\mathcal{R}$ there is a construction of 
the corresponding {\em cell module}. The later is
obtained by restricting the action of $\cS$ to the
full subcategory of $\mathcal{O}_0$ consisting of
all modules $M$ admitting a presentation
$X_1\to X_0\tto M$, where every indecomposable direct
summand of both $X_0$ and $X_1$ is isomorphic to 
$\theta_w L(d)$, where $w\in \mathcal{R}$ and
$d$ is the Duflo involution in $\mathcal{R}$.
Similarly to the proof of Theorem~\ref{thm61}
one shows that this cell module is equivalent to
the cell $2$-representation $\mathbf{C}_{\mathcal{R}}$
of $\cS$.

Let $\mathcal{Q}$ be a strongly regular two-sided cell for $\cS$. 
In this case from \cite[Theorem~5.3]{Ne} it follows 
that the condition \eqref{eq62} is satisfied for
$\mathcal{Q}$. Hence from Theorem~\ref{thm61} we obtain 
that cell $2$-representations of $\cS$ for right cells 
inside a given two-sided cell are equivalent. This reproves,
strengthens and extends the similar result
\cite[Theorem~18]{MS2}, originally proved in the case 
$\mathfrak{g}\cong\mathfrak{sl}_n$.

\subsection{Projective functors between singular blocks of $\mathcal{O}$}\label{s6.2}

The $2$-category $\cS_{\mathfrak{g}}$ from the
previous subsection admits the 
following natural generalization. For every parabolic
subalgebra $\mathfrak{p}$ of $\mathfrak{g}$
containing the Borel subalgebra 
$\mathfrak{b}=\mathfrak{h}\oplus \mathfrak{n}_+$
let $W_{\mathfrak{p}}\subset W$ be the corresponding
parabolic subgroup. Fix some dominant and integral
weight $\lambda_{\mathfrak{p}}$ such that 
$W_{\mathfrak{p}}$ coincides with the stabilizer of 
$\lambda_{\mathfrak{p}}$ with respect to the dot action
(to show the connection with the previous subsection
we take $\lambda_{\mathfrak{b}}=0$). Let
$\mathcal{O}_{\lambda_{\mathfrak{p}}}$ denote 
the corresponding block of the category $\mathcal{O}$.

Consider the $2$-category 
$\cS^{\mathrm{sing}}=\cS_{\mathfrak{g}}^{\mathrm{sing}}$
defined as follows: its objects are the categories 
$\mathcal{O}_{\lambda_{\mathfrak{p}}}$, where 
$\mathfrak{p}$ runs through the (finite!) set of
parabolic subalgebras of $\mathfrak{g}$
containing $\mathfrak{b}$,
its $1$-morphisms are all projective functors
between these blocks, its $2$-morphisms are all
natural transformations of functors. Similarly to
the previous subsection, the $2$-category 
$\cS^{\mathrm{sing}}$ is a fiat-category.
The category $\cS$ from the previous subsection is
just the full subcategory of $\cS^{\mathrm{sing}}$
with the object $\mathcal{O}_0$.
A deformed version of $\cS^{\mathrm{sing}}$ (which has 
infinite-dimensional spaces of $2$-morphisms and
hence is not fiat) was considered in \cite{Wi}.

Let us describe in more detail the structure 
of $\cS^{\mathrm{sing}}$ in the smallest nontrivial
case of $\mathfrak{g}=\mathfrak{sl}_2$.
In this case we have two parabolic subalgebras,
namely $\mathfrak{b}$ and $\mathfrak{g}$. Using
the usual identification of $\mathfrak{h}$ with
$\mathbb{C}$ we set $\lambda_{\mathfrak{b}}=0$
and $\lambda_{\mathfrak{g}}=-1$. The objects of
$\cS^{\mathrm{sing}}$ are thus $\mathtt{i}=\mathcal{O}_0$ 
and $\mathtt{j}=\mathcal{O}_{-1}$.

The category $\cS^{\mathrm{sing}}(\mathtt{j},\mathtt{j})$ contains a 
unique (up to isomorphism) indecomposable object, namely 
$\mathbbm{1}_{\mathtt{j}}$, the identity functor on $\mathtt{j}$. 
The category $\cS^{\mathrm{sing}}(\mathtt{i},\mathtt{j})$ contains a 
unique (up to isomorphism) indecomposable object, namely the functor 
$\theta^{\mathrm{on}}$ of translation onto the wall. 
The category $\cS^{\mathrm{sing}}(\mathtt{j},\mathtt{i})$
contains a unique (up to isomorphism) indecomposable object, namely the
functor $\theta^{\mathrm{out}}$ of translation out of the wall. 
The category $\cS^{\mathrm{sing}}(\mathtt{i},\mathtt{i})$
contains exactly two (up to isomorphism) non-isomorphic indecomposable 
objects, namely the identity functor $\mathbbm{1}_{\mathtt{i}}$ and the 
functor $\theta:=\theta^{\mathrm{out}}\circ\theta^{\mathrm{on}}$ of 
translation through the wall. 

It is easy to see that there are
exactly two two-sided cells: one containing only the
functor $\mathbbm{1}_{\mathtt{i}}$, and the other one
containing all other functors. The right cells of the
latter two-sided cell are $\{\mathbbm{1}_{\mathtt{j}},
\theta^{\mathrm{out}}\}$ and $\{\theta,
\theta^{\mathrm{on}}\}$. The left cells of the
latter two-sided cell are $\{\mathbbm{1}_{\mathtt{j}},
\theta^{\mathrm{on}}\}$ and $\{\theta,
\theta^{\mathrm{out}}\}$. All cells are strongly
regular. The values of the function
$m_{\mathrm{F},\mathrm{F}}$ from \eqref{eq62}
are given by:
\begin{displaymath}
\begin{array}{c||c|c|c|c|c}
\mathrm{F}&\mathbbm{1}_{\mathtt{i}}&
\mathbbm{1}_{\mathtt{j}}&\theta^{\mathrm{on}}&
\theta^{\mathrm{out}}&\theta\\ \hline
m_{\mathrm{F},\mathrm{F}}&1&1&1&2&2
\end{array}
\end{displaymath}
In particular, the condition \eqref{eq62} is satisfied.

The cell $2$-representation corresponding to the right 
cell $\{\mathbbm{1}_{\mathtt{i}}\}$ is given by the
following picture (with the obvious action of the
identity $1$-morphisms):
\begin{displaymath}
\xymatrix{ 
\mathbb{C}\text{-}\mathrm{mod}
\ar@(ul,ur)[]^{\theta=0}
\ar@/^1pc/[rrrr]^{\theta^{\mathrm{on}}=0}
&&&& 0
\ar@/^1pc/[llll]^{\theta^{\mathrm{out}}=0}.
}
\end{displaymath}

By Theorem~\ref{thm61}, the cell $2$-representations
for the right cells $\{\mathbbm{1}_{\mathtt{j}},
\theta^{\mathrm{out}}\}$ and $\{\theta,
\theta^{\mathrm{on}}\}$ are equivalent and strongly
simple. Consider the algebra $D:=\mathbb{C}(x)/(x^2)$
of dual numbers with the fixed subalgebra $\mathbb{C}$ consisting of
scalars. The cell $2$-representation for the right cell
$\{\mathbbm{1}_{\mathtt{j}}, \theta^{\mathrm{out}}\}$
is given (up to isomorphism of functors) by the following picture: 
\begin{displaymath}
\xymatrix{ 
D\text{-}\mathrm{mod}
\ar@(ul,ur)[]^{\theta=D\otimes_{\mathbb{}}{}_-}
\ar@/^1pc/[rrrr]^{\theta^{\mathrm{on}}=
\mathrm{Res}^D_{\mathbb{C}}}
&&&& \mathbb{C}\text{-}\mathrm{mod}
\ar@/^1pc/[llll]^{\theta^{\mathrm{out}}=
\mathrm{Ind}^D_{\mathbb{C}}}.
}
\end{displaymath}

\subsection{Projective functors 
for finite-dimensional algebras}\label{s6.3}

The last example admits a straightforward 
abstract generalization outside category $\mathcal{O}$.
Let $A=A_{\mathtt{1}}\oplus A_{\mathtt{2}}\oplus 
\cdots\oplus A_{\mathtt{k}}$ be a weakly symmetric 
self-injective finite-dimensional algebra over 
an algebraically closed field $\Bbbk$
with a fixed decomposition into a direct sum of
connected components (here {\em weakly symmetric} means 
that the top and the socle of every projective module
are isomorphic). Let $\cC_A$ denote the $2$-category
with objects $\mathtt{1},\mathtt{2},\dots,\mathtt{k}$,
which we identify with the corresponding
$A_{\mathtt{i}}\text{-}\mathrm{mod}$. For
$\mathtt{i},\mathtt{j}\in \{\mathtt{1},\mathtt{2},
\dots,\mathtt{k}\}$ define $\cC_A(\mathtt{i},\mathtt{j})$
as the full fully additive subcategory of the category of 
all functors from $A_{\mathtt{i}}\text{-}\mathrm{mod}$
to $A_{\mathtt{j}}\text{-}\mathrm{mod}$, generated by all 
functors isomorphic to tensoring with $A_{\mathtt{i}}$ 
(in the case $\mathtt{i}=\mathtt{j}$) and tensoring with
all projective $A_{\mathtt{j}}\text{-}A_{\mathtt{i}}$
bimodules (i.e. bimodules of the form
$A_{\mathtt{j}}e\otimes_{\Bbbk}fA_{\mathtt{i}}$ for some idempotents
$e\in A_{\mathtt{j}}$ and $f\in A_{\mathtt{i}}$) for
all $\mathtt{i}$ and $\mathtt{j}$.
Functors, isomorphic to tensoring with projective bimodules
will be called {\em projective functors}. 

\begin{lemma}\label{lem85}
The category $\cC_A$ is a fiat category.
\end{lemma}

\begin{proof}
The only nontrivial condition to check is that
the left and the right adjoints of a projective functor
are again projective and isomorphic. For any 
$A$-module $M$ and idempotents $e,f\in A$, 
using adjunction and projectivity of $fA$ we have
\begin{displaymath}
\begin{array}{rcl}
\mathrm{Hom}_A(Ae\otimes_{\Bbbk}fA,M)&=&
\mathrm{Hom}_{\Bbbk}(fA,\mathrm{Hom}_A(Ae,M))\\
&=&\mathrm{Hom}_{\Bbbk}(fA,eM)\\
&=&\mathrm{Hom}_{\Bbbk}(fA,eA\otimes_A M)\\
&=&\mathrm{Hom}_{\Bbbk}(fA,\Bbbk)
\otimes_{\Bbbk}eA\otimes_A M\\
&=&(fA)^*\otimes_{\Bbbk}eA\otimes_A M\\
\end{array}
\end{displaymath}
Since $A$ is self-injective, $(fA)^*$ is projective.
Since $A$ is weakly symmetric, $(fA)^*\cong Af$.
This implies that tensoring with 
$Af\otimes_{\Bbbk}eA$ is right adjoint to tensoring 
with $Ae\otimes_{\Bbbk}fA$. The claim follows.
\end{proof}

The category $\cC_A$ has a unique maximal two-sided cell
$\mathcal{Q}$ consisting of all projective functors. 
This cell is regular. Right and left cells 
inside $\mathcal{Q}$ are given
by fixing primitive idempotents occurring on the left
and on the right in projective functors, respectively.
In particular, they are in bijection with simple
$A\text{-}A$-bimodules and hence $\mathcal{Q}$ is strongly regular.
The value of the function $m_{\mathrm{F},\mathrm{F}}$
on $Ae\otimes_{\Bbbk}fA$ is
given by the dimension of $eA\otimes_A Ae\cong eAe$, in particular,
the function $m_{\mathrm{F},\mathrm{F}}$ 
is constant on left cells. From
Theorem~\ref{thm61} we thus again obtain that all
cell $2$-representations of $\cC$ corresponding to
right cells in $\mathcal{Q}$ are strongly simple and
isomorphic.

The category $\cS^{\mathrm{sing}}_{\mathfrak{sl}_2}$ 
from the previous subsection is obtained by taking
$\mathtt{k}=\mathtt{2}$, $A_{\mathtt{1}}=\mathbb{C}$
and $A_{\mathtt{2}}=D$. In the general case we have 
the following:

\begin{proposition}
Let $\cC$ be a fiat category, $\mathcal{Q}$ a strongly
regular two-sided cell of $\mathcal{Q}$ and $\mathcal{R}$ a 
right cell in $\mathcal{Q}$. For $\mathtt{i}\in\cC$
let $A_{\mathtt{i}}$ be such that 
$\mathbf{C}_{\mathcal{R}}(\mathtt{i})\cong
A_{\mathtt{i}}\text{-}\mathrm{mod}$ and
$A=\oplus_{\mathtt{i}\in\cC}A_{\mathtt{i}}$. 
Assume that the condition \eqref{eq62} is satisfied.
Then $\mathbf{C}_{\mathcal{R}}$ gives rise to a $2$-functor 
from $\cC_{\mathcal{Q}}$ to $\cC_A$.
\end{proposition}

\begin{proof}
We identify $\mathbf{C}_{\mathcal{R}}(\mathtt{i})$ 
with $A_{\mathtt{i}}\text{-}\mathrm{mod}$. That
$A$ is self-injective follows from
Corollary~\ref{cor56}. That
$A$ is weakly symmetric follows by adjunction from
Lemma~\ref{lem11} and strong regularity of $\mathcal{Q}$.
Hence, to prove the claim we only need to show that for any 
$\mathrm{F}\in \mathcal{Q}$ the functor
$\mathbf{C}_{\mathcal{R}}(\mathrm{F})$ is a projective
endofunctor of $A\text{-}\mathrm{mod}$.

As $\mathbf{C}_{\mathcal{R}}(\mathrm{F})$ is exact, it is
given by tensoring with some bimodule, say $B$. Since
$\mathbf{C}_{\mathcal{R}}(\mathrm{F})$ kills all
simples but one, say $L$, and sends $L$ to an
indecomposable projective, say $P$ (by Theorem~\ref{thm61}),
the bimodule $B$ has simple top (as a bimodule)
and hence is a quotient of some projective bimodule.
 
By exactness of $\mathbf{C}_{\mathcal{R}}(\mathrm{F})$, 
the dimension of $B$ equals the dimension of $P$
times the multiplicity of $L$ in $A$. This is
exactly the dimension of the corresponding indecomposable
projective bimodule. The claim follows.
\end{proof}

\vspace{0.3cm}

\noindent
Volodymyr Mazorchuk, Department of Mathematics, Uppsala University,
Box 480, 751 06, Uppsala, SWEDEN, {\tt mazor\symbol{64}math.uu.se};
http://www.math.uu.se/$\tilde{\hspace{1mm}}$mazor/.
\vspace{0.1cm}

\noindent
Vanessa Miemietz, School of Mathematics, University of East Anglia,
Norwich, UK, NR4 7TJ, {\tt v.miemietz\symbol{64}uea.ac.uk};
http://www.uea.ac.uk/$\tilde{\hspace{1mm}}$byr09xgu/.

\end{document}